\documentclass{amsart}

\usepackage{amscd,amsthm,amssymb,amsfonts,amsmath,euscript}

\theoremstyle{plain}
\newtheorem{thm}{Theorem}[section]
\newtheorem{lemma}[thm]{Lemma}
\newtheorem{prop}[thm]{Proposition}
\newtheorem{cor}[thm]{Corollary}

\theoremstyle{definition}

\theoremstyle{remark}
\newtheorem{remark}[thm]{Remark}

\newcommand{\nc}{\newcommand}

\def\makeop#1{\expandafter\def\csname#1\endcsname
  {\mathop{\rm #1}\nolimits}\ignorespaces}
\makeop{Hom}   \makeop{End}   \makeop{Aut}   \makeop{Isom}  \makeop{Pic} 
\makeop{Gal}   \makeop{ord}   \makeop{Char}  \makeop{Div}   \makeop{Lie} 
\makeop{PGL}   \makeop{Corr}  \makeop{PSL}   \makeop{sgn}   \makeop{Spf}
\makeop{Spec}  \makeop{Tr}    \makeop{Nr}    \makeop{Fr}    \makeop{disc}
\makeop{Proj}  \makeop{supp}  \makeop{ker}   \makeop{im}    \makeop{dom}
\makeop{coker} \makeop{Stab}  \makeop{SO}    \makeop{SL}    \makeop{SL}
\makeop{Cl}    \makeop{cond}  \makeop{Br}    \makeop{inv}   \makeop{rank}
\makeop{id}    \makeop{Fil}   \makeop{Frac}  \makeop{GL}    \makeop{SU}
\makeop{Trd}   \makeop{Sp}    \makeop{Tr}    \makeop{Trd}   \makeop{diag}
\makeop{Res}   \makeop{ind}   \makeop{depth} \makeop{Tr}    \makeop{st}
\makeop{Ad}    \makeop{Int}   \makeop{tr}    \makeop{Sym}   \makeop{can}
\makeop{length}\makeop{SO}    \makeop{torsion} \makeop{GSp} \makeop{Ker}
\def\makebb#1{\expandafter\def
  \csname bb#1\endcsname{{\mathbb{#1}}}\ignorespaces}
\def\makebf#1{\expandafter\def\csname bf#1\endcsname{{\bf
      #1}}\ignorespaces} 
\def\makegr#1{\expandafter\def
  \csname gr#1\endcsname{{\mathfrak{#1}}}\ignorespaces}
\def\makescr#1{\expandafter\def
  \csname scr#1\endcsname{{\EuScript{#1}}}\ignorespaces}
\def\makecal#1{\expandafter\def\csname cal#1\endcsname{{\mathcal
      #1}}\ignorespaces} 

\def\doLetters#1{#1A #1B #1C #1D #1E #1F #1G #1H #1I #1J #1K #1L #1M
                 #1N #1O #1P #1Q #1R #1S #1T #1U #1V #1W #1X #1Y #1Z}
\def\doletters#1{#1a #1b #1c #1d #1e #1f #1g #1h #1i #1j #1k #1l #1m
                 #1n #1o #1p #1q #1r #1s #1t #1u #1v #1w #1x #1y #1z}
\doLetters\makebb   \doLetters\makecal  \doLetters\makebf
\doLetters\makescr 
\doletters\makebf   \doLetters\makegr   \doletters\makegr
     \def\qed{\qedmark\medbreak}%
\def\qedmark{{\enspace\vrule height 6pt width 5pt depth 1.5pt}}%
    \def\setminus{\smallsetminus}

\normalsize

\makeop{Bl}

\def\Spec{{\rm Spec}}

\def\Fpbar{\overline{\bbF}_p}

\def\Qpbar{\overline{{\bbQ}_p}}

\def\Qbar{\overline{\bbQ}}

\newcommand{\Z}{\mathbb Z}
\newcommand{\Q}{\mathbb Q}
\newcommand{\R}{\mathbb R}
\newcommand{\C}{\mathbb C}
\newcommand{\A}{\mathbb A}    
\renewcommand{\O}{\mathcal O} 
\newcommand{\F}{\mathbb F}

\newcommand{\pr}{\indent }

\newcommand{\<}{\langle}   
\renewcommand{\>}{\rangle} 

\newcommand{\isoto}{\stackrel{\sim}{\to}}
\nc{\embed}{\hookrightarrow}

\newcommand{\ch}{characteristic }
\newcommand{\ac}{algebraically closed }
\newcommand{\dieu}{Dieudonn\'{e} }

\nc{\ol}{\overline}
\nc{\wt}{\widetilde}
\nc{\opp}{\mathrm{opp}}
\def\ul{\underline}

\makeop{Ram}
\makeop{Rep}

\begin{document}
\renewcommand{\thefootnote}{\fnsymbol{footnote}}
\setcounter{footnote}{-1}
\numberwithin{equation}{section}

\title[Hilbert-Blumenthal moduli spaces]{Irreducibility of the
  Hilbert-Blumenthal moduli spaces with parahoric level structure}

\author{Chia-Fu Yu}
\address{
Institute of Mathematics \\
Academia Sinica \\
128 Academia Rd.~Sec.~2, Nankang\\ 
Taipei, Taiwan \\ and NCTS (Taipei Office)}
\address{
Max-Planck-Institut f\"ur Mathematik \\
Vivatsgasse 7 \\
Bonn, 53111\\ 
Germany}
\email{chiafu@math.sinica.edu.tw}

\date{\today.  The research is partially supported by NSC
 96-2115-M-001-001, the Orchid project}


\begin{abstract}
  We determine the number of irreducible components of the
  reduction mod $p$ of any Hilbert-Blumenthal moduli space 
  with a parahoric level structure, where 
  $p$ is unramified in the totally real field. 
\end{abstract} 

\maketitle


\def\k{\mathrm k}


 

\section{Introduction}
\label{sec:01}

In their 1984 paper \cite{brylinski-labesse}, Brylinski and Labesse
computed the L-factors of Hilbert-Blumenthal moduli spaces for almost
all good places. By that time the arithmetic minimal compactification
was not known. In \cite{chai:amchb} Chai furnished the desired minimal
compactification by observing that Rapoport's arithmetic toroidal
compactification \cite{rapoport:thesis} plays the crucial role. Thus,
the results of Brylinski and Labesse have been improved for all good
places (see \cite[p.~137]{faltings-chai}).  
A next task is to treat the case where $p$ is unramified
and the level group $K_p$ at $p$ is a standard
Iwahori subgroup. This moduli space is studied in Stamm \cite{stamm},  
following the works of Zink \cite{zink:thesis} and of Rapoport-Zink 
\cite{rapoport-zink}. Several local
properties on geometry as well as fine global descriptions of the
surface case have been obtained in \cite{stamm}.
In this paper we settle a global problem
concerning the irreducibility in this moduli space. \\

Let $p$ be a fixed rational prime. Let $F$ be a totally real number 
field of degree $g$ and $O_F$ the ring of integers. Let $n\ge 3$ be a
prime-to-$p$ integer. Choose a primitive $n$th root $\zeta_n$ of
unity in 
$\Qbar\subset \C$ and an embedding $\Qbar\hookrightarrow \Qbar_p$.  
Let $(L,L^+)$ be a rank one projective $O_F$-module with a notion of
positivity. Let $\calM_{(L,L^+),n}$ denote the moduli space over
$\Z_{(p)}[\zeta_n]$ that parametrizes equivalence classes of objects $\ul
A=(A,i,\iota,\eta)$ over a locally Noetherian 
$\Z_{(p)}[\zeta_n]$-scheme $S$, where
\begin{itemize}
\item $A$ is an abelian scheme of relative dimension $g$;
\item $\iota:O_F \to \End_S(A)$ is a ring monomorphism;
\item $i:(L,L^+)_S\to (\calP(A),\calP(A)^+)$ is a morphism
  of {\'e}tale sheaves such that the induced morphism
  \begin{equation}
    \label{eq:101}
    L\otimes_{O_F} A\to A^t 
  \end{equation}
  is an isomorphism, where $(\calP(A),\calP(A)^+)$ is the polarization
  sheaf of $A$ (see \cite{deligne-pappas});
\item $\eta:(O_F/n O_F)_S^2\simeq A[n]$ is an $O_F$-linear isomorphism
  such that the pull back of the Weil pairing $e_{i(\lambda_0)}$ is
  the standard pairing on $(O_F/n O_F)^2$ with respect to $\zeta_n$,
  where $\lambda_0$ is any element in $L^+$ such that
  $|L/O_F\lambda_0|$ is prime to $pn$. 
\end{itemize}

It is proved in Rapoport \cite{rapoport:thesis} and 
Deligne-Pappas \cite{deligne-pappas} that
\begin{thm}[Rapoport, Deligne-Pappas]
  The fibers of $\calM_{(L,L^+),n}\to \Spec\,\Z_{(p)}[\zeta_n]$ are
  geometrically irreducible.
\end{thm}

In the paper we consider the Iwahori level structure
$\calM_{(L,L^+),\Gamma_0(p),n}$ over $\calM_{(L,L^+),n}$ where
$\calM_{(L,L^+),n}$ has good reduction at $p$. 
The goal is to determine the
set of irreducible components of the reduction
$\calM_{(L,L^+),\Gamma_0(p),n}\otimes \ol \F_p$ modulo $p$. 
We write $\Pi_0(X)$ for the set of 
irreducible components of a Noetherian scheme $X$. Thus, if $X$ is a
scheme of finite type over a field $K$, then
$\Pi_0(X\otimes \ol K)$ is in bijection with the set of geometrically
irreducible components of $X$. The latter is the case considered in this
paper. \\

Assume that $p$ is unramified in $F$. Let
$\calM_{(L,L^+),\Gamma_0(p),n}$ denote the moduli space over
$\Z_{(p)}[\zeta_n]$ that
parametrizes equivalence classes of 
objects $(A,i,\iota,H,\eta)$, where 
\begin{itemize}
\item $(A,i,\iota,\eta)$ is in $\calM_{(L,L^+),n}$, and 
\item $H\subset A[p]$ is a finite flat rank $p^g$
subgroup scheme which is invariant under the action of $O_F$ and
maximally isotropic with respect to the Weil pairing
$e_{i(\lambda_0)}$ as above. 
\end{itemize}

Write $\calM:=\calM_{(L,L^+),n}\otimes \ol \F_p$ and
$\calM_{\Gamma_0(p)}:=\calM_{(L,L^+),\Gamma_0(p),n}\otimes \ol
\F_p$ throughout this paper. We will state our main results concerning  
the number $|\Pi_0(\calM_{\Gamma_0(p)})|$ in the next
section. We describe them together with background and methods.
See Theorem~\ref{27} and Theorem~\ref{59} for the precise statement.\\ 

The method in this paper is completely different from that used in
\cite{yu:parahoric} for the Siegel moduli spaces. In the previous paper the
proof is based on the Faltings-Chai theorem on the $p$-adic monodromy
for the ordinary locus \cite{faltings-chai} and a theorem proved by
Ng{\^o} and Genestier \cite{ngo-genestier:alcoves} that the ordinary
locus is dense in the parahoric level moduli spaces. The latter is
obtained by analyzing the so called 
{\it Kottwitz-Rapoport} stratification
introduced in \cite{kottwitz-rapoport:alcoves}
(cf. \cite{ngo-genestier:alcoves}). \\ 

For the present situation, the ordinary locus is no longer dense, as 
has been pointed out in Stamm \cite{stamm} in the surface
case. Thus Ribet's $p$-adic monodromy result \cite{ribet:hmf} can only   
conclude the irreducibility for ordinary components.
One may need to establish the surjectivity result of the naive 
$p$-adic monodromy for smaller strata in $\calM$ defined by certain
$p$-adic invariant, which is not
available yet. However, even though we can prove these $p$-adic 
monodromy results, the standard $p$-adic monodromy argument is still
insufficient to conclude the irreducibility of non-ordinary
components. Indeed, the naive $p$-adic monodromy result implies the
connectedness of covers of a stratum arising from the etale
$p^m$-torsion part, and these covers are always finite.  
On the other hand, a possibly irreducible non-ordinary component of
$\calM_{\Gamma_0(p)}$ is necessarily having positive dimensional
fibers over a stratum in $\calM$, and hence cannot be dominated 
by a finite cover.  
Besides the non-density, we do not have yet geometric properties for
Kottwitz-Rapoport strata of $\calM_{\Gamma_0(p)}$ along the direction
of work of Ng\^o-Genestier \cite{ngo-genestier:alcoves}. \\ 

To overcome these new difficulties, we stratify the moduli
space by a suitable $p$-adic invariant: 
\[ \calM_{\Gamma_0(p)}=\coprod_\alpha \calM_{\Gamma_0(p),\alpha}. \]
Then we study the corresponding {\it discrete
Hecke orbit problem}, namely asking whether the prime-to-$p$ Hecke
correspondences operate transitively on the set
$\Pi_0(\calM_{\Gamma_0(p),\alpha})$. This discrete Hecke orbit
problem, though itself does not have an affirmative answer, 
can be refined through 
the computation of the fibers of the stratified morphism
$f_\alpha: \calM_{\Gamma_0(p),\alpha}\to\calM_{\alpha}$, and is
reduced to the discrete Hecke orbit problem for the set
$\Pi_0(\calM_\alpha)$ of irreducible
components of the base. 
The former one can be done using \dieu calculus, for which the present   
computation (see Sections 3 and 4) is largely based on 
the work \cite{yu:thesis}. \\  

The next crucial ingredient is Chai's monodromy theorem on Hecke
invariant subvarieties. This is a global method which may be regarded
as the counterpart of the $p$-adic monodromy method.   
Its original form for Siegel moduli spaces
is developed by Chai \cite{chai:mono}. Chai's method works for good
reductions of any modular variety of PEL-type, 
with the modification where 
the reductive group in the Shimura
input data should be replaced by the simply-connected cover of 
its derived group \cite[p.~291]{chai:mono}.
%
We supply the proof due to Chai in Section~\ref{sec:06} for the
reader's convenience.   
This ingredient enables us to confirm
the irreducibility in the non-supersingular contribution (components
that contains non-supersingular points).  
To treat the remaining supersingular contribution (components that are
entirely contained in the supersingular locus), 
the tool is essentially the result that the Tamagawa
number is one for semi-simple, simply-connected algebraic groups 
\cite{kottwitz:tamagawa}. 
The present cases heavily rely on the computations of the {\it
  geometric mass formula} in \cite{yu:mass_hb}, which are based on work
\cite{shimura:1999} of Shimura. \\

The paper is organized as follows. In Section 2 we describe the main
theorems and provide the methods and ingredients. In Section 3 we give
the proofs of the theorems. In Section 4 we treat the supersingular
contribution. To make the exposition clean and more accessible, 
we assume $p$ inert in $F$ in these sections. 
In Sections~\ref{sec:05} we show how to establish the analogous results
in the unramified situation from the inert
case. Section~\ref{sec:06} provides a proof of Chai's result on Hecke
invariant subvarieties in the Hilbert-Blumenthal moduli spaces. 
We attempted to write this as an independent
section so that the reader can read this section alone together with
Chai's well written paper \cite{chai:mono}. \\

 {\bf Acknowledgments.} The present work is greatly benefited
through the author's participation of the Hecke orbit problem with
C.-L.~Chai. He is grateful to Chai for sharing his ideas, many
discussions and his constant support. 
He also wants to thank M.~Rapoport and B.C. Ng{\^o}
for their interest and encouragement on an earlier work
\cite{yu:parahoric}. Part of this work was done while the author's 
stay at l'Institut de Mathematiques de Jussieu in Paris and
Max-Planck-Institut f\"ur Mathematik in Bonn. 
He would like to thank V.~Maillot for the invitation and
well arrangement, M.~Harris for his interest, and both institutions
for kind hospitality and excellent working environment. Finally, he
wishes to thank the referee for careful reading and helpful 
comments which improve the exposition of this paper.


\section{Statements and methods}
\label{sec:02}

\subsection{}
We keep the notation as in the previous section. Let $k$ be an \ac
field of \ch $p$. We will assume in Sections 2-4 that $p$ is 
inert in $F$.
Write 
\begin{equation}
  \label{eq:211}
  f:\calM_{\Gamma_0(p)}\to \calM
\end{equation}
for the forgetful
morphism; this is a proper surjective morphism.

We recall the alpha stratification on the moduli space $\calM$ introduced
in Goren and Oort \cite{goren-oort} and in \cite[Section 3] {yu:thesis}.
Let $W:=W(k)$ be the ring of Witt vectors over $k$ and $\sigma$ the
absolute Frobenius map on $W$. Put $\O:=O_F\otimes \Z_p$ and let 
$\scrI:=\Hom (\O,W)=\{\sigma_i\}$ be the set of embeddings, arranged
in a way that $\sigma \sigma_i=\sigma_{i+1}$ for $i\in \Z/g\Z$.
Let $\ul A=(A,\iota)$ be an abelian $O_F$-variety over $k$, 
and let $\ul M$ be the associated covariant \dieu $\O$-module. The
alpha type of $\ul A$ is defined to be
\[ \ul a(\ul A):=\ul a(\ul M):=(a_i)_{i\in \Z/g\Z}, \quad \text{where
  }\ a_i:=\dim_k
(M/(F,V)M)^i, \]
Here $(M/(F,V)M)^i$ denotes the $\sigma_i$-component of the $k$-vector
space $M/(F,V)M$. Since $M$ is a free $O_F\otimes W$-module of rank
two, each $a_i$ lies in $\{0,1,2\}$. 
If $\ul A=(A,i,\iota,\eta)$ is a point in $\calM(k)$, the alpha type
$\ul a(\ul A)$ of $\ul A$ is defined to be the alpha type of its
underlying abelian $O_F$-variety $(A,\iota)$. In this case each $a_i$
is $0$ or $1$, since ${\rm Lie} A$ satisfies the Rapoport condition. 


For each $\ul a\in \{0,1\}^\scrI$, 
let $\calM_{\ul a}$ denote the reduced subscheme of $\calM$ consisting
of points with alpha type $\ul a$. It is known \cite{goren-oort} 
that every alpha stratum $\calM_{\ul a}$ is
non-empty.  Put $\Delta:=\{0,1\}^\scrI$; this is exactly the set of alpha
types that occur in $\calM$.  
The partial order on $\Delta$ is given by $\ul a'\preceq  \ul a$ if and
only if $a_i'\ge a_i$ for all $i\in \Z/g\Z$. 
Write $|\ul a|:=\sum_{i\in \Z/g\Z} a_i$, the size of the alpha type
$\ul a$. 
It is proved in Goren and Oort \cite{goren-oort} that each 
stratum $\calM_{\ul  a}$ is smooth, 
quasi-affine, of pure dimension $g-|\ul a|$, and that the Zariski
closure of $\ol{\calM}_{\ul a}$ in $\calM$ is smooth and
\[  \ol{\calM}_{\ul a}=\cup_{\ul a'\preceq  \ul a}\, \calM_{\ul
  a'}. \\ \]  
Thus, we have a stratification of $\calM$
\[ \calM=\coprod_{\ul a\in \Delta} \calM_{\ul a} \]
by locally closed smooth subschemes $\calM_{\ul a}$, 
called the alpha stratification.
 
For each $\ul a\in \Delta$, put $\calM_{\Gamma_0(p),\ul
  a}:=\calM_{\Gamma_0(p)}\times_\calM \calM_{\ul a}$ and let 
\[ f_{\ul a}:
\calM_{\Gamma_0(p),\ul a}\to \calM_{\ul a}. \]
be the restriction of $f$ on 
$\calM_{\Gamma_0(p),\ul a}$. We decompose the moduli space 
\[ \calM_{\Gamma_0(p)}=\coprod_{\ul a\in \Delta}
\calM_{\Gamma_0(p),\ul a} \]
into locally closed subschemes.



In \cite[Section 2]{yu:thesis} an alpha 
type $\ul a=(a_i)_i\in \Delta$ is called {\it generic} 
if $a_i a_{i+1}=0$ for all $i\in \Z/g\Z$. For example, when $g=4$, the
generic alpha types are 
\[ (0,0,0,0), (1,0,0,0), (0,1,0,0), (0,0,1,0), \]
\[ (0,0,0,1), (1,0,1,0),(0,1,0,1). \]
This notion was first
introduced by Goren and Oort \cite{goren-oort} in which 
it is called {\it spaced}. It is used in order to describe the Newton
points of maximal points of E-O strata (see \cite[Theorem 5.4.11]
{goren-oort}, also see Theorem~\ref{255}). Conversely, it is proved in 
\cite[Section 6]{yu:thesis} that the alpha type of any maximal point of
a Newton stratum of $\calM$ is generic. Let $\Delta^{\rm gen}\subset
\Delta$ denote the set of generic alpha types. 


In this paper we prove 

\begin{thm}\label{22}
We have $\dim \calM_{\Gamma_0(p),\ul a}=g$ if and only if
$\ul a$ is of generic type.      
\end{thm} 

Since the moduli space $\calM_{\Gamma_0(p)}$ is equi-dimensional of
dimension $g$ \cite[Theorem 1, p.~407]{stamm}, each stratum
$\calM_{\Gamma_0(p), \ul a}$ has dimension less than or equal to
$g$. Theorem~\ref{22} asserts that the equality holds exactly when
  $\ul a$ is of generic type. 
An immediate consequence is the following

\begin{cor}\label{225}
  The image of the set of maximal points of $\calM_{\Gamma_0(p)}$ under the
  morphism $f$ (\ref{eq:211}) are exactly that of maximal points of
  all generic alpha strata $\calM_{\ul a}$.
\end{cor}
 
For each alpha type  $\ul a=(a_i)\in \Delta$, denote by $\tau(\ul
a)\subset \Z/g\Z$ 
the subset consisting of elements $i$ such that
$a_i=1$. The subset $\tau(\ul a)$ is called the {\it alpha index}
corresponding to $\ul a$. 
Write $\tau(\ul a)=\{n_1,\dots, n_a\}$ with $0\le n_i <n_{i+1}<g$ and 
put $n_{a+1}=g+n_1$. Define a function $w:\Delta \to \Z$ by 
\begin{equation}
  \label{eq:221}
  w(\ul a):=w(\tau(\ul a)):=
  \begin{cases}
    2 & \text{if $\tau(\ul a)=\emptyset$}; \\
    \prod_{j=1}^a (n_{j+1}-n_j-1) & \text{otherwise}. 
  \end{cases}
\end{equation}
It is clear that $w(\ul a)>0$ if and only if $\ul a\in \Delta^{\rm
  gen}$.

\begin{thm}\label{23}
  Let $\ul a$ be a generic alpha type.

{\rm (1)} 
   For any point $x\in \calM_{\ul a}(k)$, the fiber $f^{-1}(x)$
  has $w(\ul a)$ irreducible components of dimension
   $|\ul a|$. 

{\rm (2)} The subscheme $\calM_{\Gamma_0(p),\ul a}$ has $w(\ul a)
  |\Pi_0(\calM_{\ul a})|$  irreducible components of dimension $g$. 

\end{thm}

We remark that the fiber $f^{-1}(x)$ is not equi-dimensional in
general; see Example~\ref{34}. Similarly, some
component of $\calM_{\Gamma_0(p),\ul a}$ may has dimension less than $g$.
By Theorems~\ref{22} and \ref{23} (2), we get
\begin{equation}
  \label{eq:222}
 |\Pi_0(\calM_{\Gamma_0(p)})|=\sum_{\ul a\in \Delta^{\rm gen}}
   w(\ul a) |\Pi_0(\calM_{\ul a})|.      
\end{equation}

In the next step we consider the $\ell$-adic Hecke correspondences
operating on the set $\Pi_0(\calM_{\ul a})$ of irreducible 
components, where $\ell\not=p$ is a prime.

For any non-negative integer $m\ge 0$, let $\calH_{\ell,m}$ be the
moduli space over $\Fpbar$ that parametrizes equivalence classes of
objects $(\ul A_j=(A_{j},i_{j},\iota_{j},\eta_{j}),
 j=1,2,3; \varphi_1, \varphi_2)$ as the diagram
\[ \begin{CD}
  \ul A_1@<\varphi_1<< \ul A_3 @>\varphi_2>> \ul A_2,
\end{CD} \]
where 


\begin{itemize}
\item $\ul A_1$ and $\ul A_2$ are objects in $\calM$, and $\ul A_3$ is
  a $g$-dimensional abelian $O_F$-variety with a class of polarizations and 
  a symplectic level-$n$ structure as defined in
  Section~\ref{sec:01} but the condition (\ref{eq:101}) is not
  required; 

\item the morphisms $\varphi_1$ and $\varphi_2$ are $O_F$-linear
  isogenies of degree 
  $\ell^m$ such that $(\varphi_{j})^* i_{j}=i_3$ (pull back the
  polarizations) and $(\varphi_j)_* \eta_3= \eta_j$ (pushfoward the
  level structures) for $j=1,2$. 
\end{itemize}



Let $\calH_\ell:=\cup_{m\ge 0} \calH_{\ell,m}$. An $\ell$-adic Hecke
correspondence is given by an irreducible 
component $\calH$ of $\calH_\ell$ together with natural 
projections $\rm{pr}_1$ and $\rm{pr}_2$. A subset $Z$ of $\calM$ is
called {\it $\ell$-adic Hecke invariant} if $\rm{pr}_2
(\rm{pr}_1^{-1}(Z))\subset Z$ for any $\ell$-adic Hecke correspondence
$(\calH,\rm{pr}_1,\rm{pr}_2)$. If $Z$ is an $\ell$-adic Hecke invariant,
locally closed subset of $\calM$, then the $\ell$-adic Hecke
correspondences induce 
correspondences on the set $\Pi_0(Z)$ of irreducible components. We
call $\Pi_0(Z)$ {\it $\ell$-adic Hecke transitive} if the $\ell$-adic
Hecke correspondences operate transitively on $\Pi_0(Z)$. 
The discrete Hecke problem for any $\ell$-adic Hecke invariant
subscheme $Z$ is asking whether $\Pi_0(Z)$ is $\ell$-adic Hecke
transitive.  

\begin{thm}[Chai]\label{24}
Let $Z$ be an $\ell$-adic Hecke invariant subscheme of
$\calM$. If the set $\Pi_0(Z)$ is $\ell$-adic Hecke transitive and maximal
points of $Z$ are not supersingular, then $Z$ is irreducible.
\end{thm}

Notice that the formulation of Theorem~\ref{24} does not require
our assumption on $p$ and Theorem~\ref{24} remains valid
without this assumption (see Section~\ref{sec:06}). 

Since the alpha type is an invariant under isogenies of prime-to-$p$
degree, each alpha stratum $\calM_{\ul a}$ is $\ell$-adic Hecke
invariant. The following result is due to Goren and Oort
\cite[Corollary 4.2.4]{goren-oort}  

\begin{thm}\label{241}
  For any alpha type $\ul a$, the set $\Pi_0(\calM_{\ul a})$ is
  $\ell$-adic Hecke transitive.   
\end{thm}

An alpha stratum $\calM_{\ul a}$ is called {\it supersingular} 
if it is contained in the supersingular locus. 
This is equivalent to that all of its maximal points are supersingular. 
An alpha type $\ul a$ is called {\it supersingular} if the
corresponding stratum $\calM_{\ul a}$ is so. It follows from
Theorems~\ref{24} and~\ref{241} that

\begin{cor}\label{242}
  Any non-supersingular stratum $\calM_{\ul a}$ is irreducible. 
\end{cor}

 
\subsection{}\label{25}
It remains to treat the supersingular contribution in
(\ref{eq:222}). In the following we 
describe all supersingular strata $\calM_{\ul a}$, not just for
generic ones. This is slightly more than what we need. 

For $j=g/2$ or a non-negative integer with $0\le j\le g/2$, write
$s(j,g)$ for the slope sequence 
\[ \left \{\frac{j}{g},\dots,\frac{j}{g},
\frac{g-j}{g}\dots,\frac{g-j}{g}\right \} \] 
with each multiplicity $g$. The set of all such $s(j,g)$'s is
the set of slope sequences (or Newton polygons) which are realized by
points in $\calM$ (see \cite[Lemma 3.1 and Theorem
7.4]{yu:reduction}).  

We recall the following result in Goren and Oort \cite[Theorem
5.4.11]{goren-oort}. 

\begin{thm}\label{255}
  The generic point of each component of $\calM_{\ul a}$ has slope
  sequence $s(\lambda(\ul a),g)$ except when 
  $g$ is odd and $|\ul a|=g$, where
\[  \lambda(\ul a):=\{\, |\ul b|\,;\, \ul a \preceq \ul b \quad
  \text{and $\ul b$ is generic} \}. \]
\end{thm}

In the except case where 
$g$ is odd and $|\ul a|=g$, the alpha stratum $\calM_{\ul a}$ is
the superspecial locus. 
One can characterize easily from Theorem~\ref{255} when an alpha
stratum $\calM_{\ul a}$ is supersingular. When $\ul a$ is of generic
type, the stratum $\calM_{\ul a}$ is supersingular if and only if 
$g=2d$ is even and
$|\ul a|=d$.
They correspond to alpha types
$\ul a=(1,0,\dots,1,0)$ and $\ul a=(0,1,\dots,0,1)$. See
Subsection~\ref{42} for the case of general alpha types. 
\\

Choose and fix a non-zero element
$\lambda_0$ in $L^+$ so that $(|L/O_F\lambda_0|,np)=1$. Let $x$ be any
point in $\calM_{(L,L^+),n}(\C)$.   
One associates a skew-Hermitian $O_F$-module
$H_1(A_x(\C),\Z)$ to $(A_x,i_x(\lambda_0),\iota_x)$. The isomorphism
class of the skew-Hermitian $O_F$-module $H_1(A_x(\C),\Z)$ only
depends on the moduli space $\calM_{(L,L^+),n}$, which we write
$(V_{\Z},\<\,,\>,\iota)$. Let $G$ be
the automorphism group scheme over $\Z$ associated to 
the skew-Hermitian $O_F$-module $(V_{\Z},\<\,,\>,\iota)$, and
$\Gamma(n)$ be the kernel of the reduction map $G(\Z)\to G(\Z/n\Z)$. 
One has the complex 
uniformization
\[ \calM_{(L,L^+),n}(\C)\simeq \Gamma(n)\backslash G(\R)/
SO_2(\R)^g. \]

\begin{thm}\label{26}
  Let $\calM_{\ul a}$ be a supersingular stratum.

{\rm (1)} If $g$ is odd, then $\calM_{\ul a}$ consists of all
superspecial points and 
\begin{equation}\label{eq:251}
  |\calM_{\ul a}(k)|=
  [G(\Z):\Gamma(n)]\cdot   
  \left [ \frac{-1}{2} \right ]^g \cdot \zeta_F(-1)\cdot (p^g-1).
\end{equation}

{\rm (2)} If $g$ is even and $|\ul a|=g$, then $\calM_{\ul a}$
consists of all superspecial points and 
\begin{equation}\label{eq:252}
  |\calM_{\ul a}(k)|=
  [G(\Z):\Gamma(n)]\cdot   
  \left [ \frac{-1}{2} \right ]^g \cdot \zeta_F(-1)\cdot (p^g+1).
\end{equation} 
{\rm (3)} If $g$ is even and $|\ul a|\not= g$, then any irreducible
component of $\ol{\calM}_{\ul a}$ is isomorphic to $(\bfP^1)^{g-|\ul a|}$
and 
\begin{equation}\label{eq:253}
  |\Pi_0(\calM_{\ul a})|=
  [G(\Z):\Gamma(n)]\cdot   
  \left [ \frac{-1}{2} \right ]^g \cdot \zeta_F(-1).
\end{equation} 
\end{thm}

Let $\Delta^{\rm gen}_{\rm ss}\subset \Delta^{\rm gen} $
denote the subset of supersingular generic alpha types. If $g$ is odd,
then $\Delta^{\rm gen}_{\rm ss}$ is empty; if $\ul a\in \Delta^{\rm
  gen}_{\rm ss}$, then $w(\ul a)=1$. By Corollary~\ref{242}, 
Theorem~\ref{26} (3) and (\ref{eq:222}), we get 

\begin{thm}\label{27} Notation as before. 
Assume that $p$ is inert in $F$. Then
\[ |\Pi_0(\calM_{\Gamma_0(p)})|=
\begin{cases}
\left [ \sum_{\ul a\in\Delta^{\rm gen} \setminus 
\Delta^{\rm gen}_{\rm
  ss}} 
  w(\ul a)\right ] +2
  [G(\Z):\Gamma(n)]\cdot    
  \left [ \frac{-1}{2} \right ]^g \cdot \zeta_F(-1) & \text{if $g$ is
  even;} \\ \sum_{\ul a\in\Delta^{\rm gen}} w(\ul a) & \text{if $g$
  is odd.} 
\end{cases} \]
\end{thm}

The following is an elementary combinatorial result.


\begin{lemma}\label{28}
  For any subset $\tau$ of $\Z/g\Z$, let $w(\tau)$ be as in
  (\ref{eq:221}). One has
\[ \sum_{\tau\subset \Z/g\Z} w(\tau)=2^g. \]
\end{lemma}

Since $w(\ul a)=0$ for non-generic alpha types $\ul a$ and $w(\ul
a)=1$ for $\ul a\in \Delta^{\rm gen}_{\rm ss}$ (note that $\Delta^{\rm
  gen}_{\rm ss}$ is empty if $g$ is odd), we rewrite the formula
in Theorem~\ref{27} as follows

\[ |\Pi_0(\calM_{\Gamma_0(p)})|=\sum_{\ul a\in\Delta^{}} w(\ul a)+
\sum_{\ul a\in \Delta^{\rm
    gen}_{\rm ss}}  \left \{ [G(\Z):\Gamma(n)]\cdot    
  \left [ \frac{-1}{2} \right ]^g \cdot \zeta_F(-1) - 1 \right \}\]
Using Lemma~\ref{28}, Theorem~\ref{27} is rephrased as 

\begin{thm}\label{211}
Assume that $p$ is inert in $F$. Then
\begin{equation}
  \label{eq:2111}
  |\Pi_0(\calM_{\Gamma_0(p)})|=2^g+\sum_{\ul a\in \Delta^{\rm
    gen}_{\rm ss}}  \left \{ [G(\Z):\Gamma(n)]\cdot    
  \left [ \frac{-1}{2} \right ]^g \cdot \zeta_F(-1) - 1 \right \}.
\end{equation}
\end{thm}

See a formula for $|\Pi_0(\calM_{\Gamma_0(p)})|$
when $p$ is unramified in Section~\ref{sec:05}. \\

In the following, we determine the slope sequence (or Newton polygon) 
of the generic point of each irreducible component of
$\calM_{\Gamma_0(p)}$.    

\begin{thm}\label{212} \

{\rm (1)} Let $\eta$ be a maximal point of
$\calM_{\Gamma_0(p)}$. Put $j=|\ul a(f(\eta))|$, the $a$-number of its
image $f(\eta)$.
Then the slope sequence of $\eta$ is equal to $s(j, g)$. 

{\rm (2)} 
For each non-negative integer $j$ with $0\le j< \lceil
\frac{g}{2}\rceil$, the moduli space $\calM_{\Gamma_0(p)}$ has exactly
$2 \begin{pmatrix} 
    g\\2j
  \end{pmatrix}$ irreducible components whose maximal point has
  slope sequence $s(j,g)$.

{\rm (3)} The moduli space $\calM_{\Gamma_0(p)}$ has  
\[ |\Delta^{\rm gen}_{\rm ss}|\cdot [G(\Z):\Gamma(n)]\cdot    
  \left [ \frac{-1}{2} \right ]^g \cdot \zeta_F(-1) \]
supersingular irreducible components.
\end{thm}
 
Note that statement (2) deals with non-supersingular slope sequences.




\section{Proof of Theorems~\ref{22} and \ref{23}}
\label{sec:03}

\subsection{}
Let $f:\calM_{\Gamma_0(p)}\to \calM$ be the forgetful morphism, and
let $x=(A,i_A,\iota_A,\eta_A)$ be a point in $\calM_{\ul
  a}(k)$. Choose a separable $O_F$-linear polarization
$\lambda_A=i_A(\lambda_0)$ on $A$. Each point in $f^{-1}(x)$ is given
by an $O_F$-invariant finite subgroup scheme $H$ of $A$ of rank $p^g$
which is maximally isotropic with respect to the Weil pairing
$e_{\lambda_A}$. Then there is an $O_F$-linear polarization
$\lambda_B$, necessarily separable,  on $B:=A/H$ such that the pull back
$\pi^*\lambda_B$ is equal to $p\lambda_A$. Denote by $M^*(A)$ the
classical contravariant \dieu module of $A$. We have an $\O$-invariant
\dieu submodule $M^*(B)$ of $M^*(A)$ such that 
\[ M^*(A)/M^*(B)\cong k\oplus\dots \oplus k, \quad \text{and } 
\<\, ,\,\>_{M^*(A)}=p\, \<\, ,\,\>_{M^*(B)}. \]

Note that $M^*(A)$ is canonically isomorphic to the dual $M(A)^t$
of the covariant \dieu module $M(A)$. We also know that $\ul
a(M(A)^t)=\ul a (M(A))$ (see \cite[Lemma 8.1]{yu:thesis}). Put
$M_0:=M^*(A)$ and let $\tau:=\tau(\ul a)$ be corresponding alpha index
as in Section~\ref{sec:02}. Let 
$\calX_\tau$ be the space of \dieu $\O$-submodules 
$M$ of $M_0$ such that  
\[ M_0/M\cong k\oplus \dots \oplus k. \]
We regard $\calX_\tau$ as a scheme over $k$ with the reduced  
structure. 
For any point $M$ in $\calX_\tau$, it is clear that the induced
$k$-valued pairing $\<\,,\,\>$ is trivial on $M/p M_0$. Therefore
there is a polarized abelian
$O_F$-variety $\ul B=(B,\lambda_B,\iota_B)$ and an $O_F$-linear isogeny
$\pi:\ul A\to \ul B$ such that $\pi^*\lambda_B=p\lambda_A$ and
$M^*(B)=M$. 
This establishes
\begin{lemma}\label{32}
  The map $(\ul A, H)\mapsto M^*(A/H)$ gives rise to an isomorphism
  $\xi_{x}: f^{-1}(x)_{\rm 
  red}\simeq \calX_\tau$, where $f^{-1}(x)_{\rm red}$ is the reduced
  subscheme underlying the fiber $f^{-1}(x)$.
\end{lemma}


\begin{lemma}\label{33}
The scheme $\calX_\tau$ is isomorphic to
the subscheme of $(\bfP^1)^g=\{ ([s_i:t_i])_{i\in\Z/g\Z} \}$ defined
by the 
  equations $t_{i-1} s_i=0$ for $i\not\in \tau$ and $t_{i-1}t_i=0$ for
  $i\in \tau$.   
\end{lemma}
\begin{proof}
  A point in $\calX_\tau(k)$ is represented by a $k$-subspace $\ol M$
  of $\ol M_0:=M_0/pM_0$ such that $F(\ol M)\subset \ol M$, $V(\ol
  M)\subset \ol M$, and $\dim_k \ol M^i=1$ for each
  $i\in \Z/g\Z$. Hence  $\calX_\tau$ is a closed subscheme of $(\bfP^1)^g$.  
  Choose a basis $\{ X_i,Y_i\}$ for $M_0$ 
  \cite[Proposition~4.2]{yu:thesis} such that
\begin{align*}
  FX_{i-1}&=\begin{cases}
         X_{i} & \text{if } i\not\in \tau;\\
         Y_{i}+p c_{i} X_{i} & \text{if }i\in \tau;
         \end{cases} \quad 
  FY_{i-1}=\begin{cases}
         p Y_{i} & \text{if } i\not\in \tau;\\
         p X_{i} & \text{if } i\in \tau;
         \end{cases}
\end{align*}
where $c_i$ are some elements of $W(k)$ for $i\in \tau$. (There should
be no confusion on our notation for the Frobenius map and the totally
real field.) 
Let $P=([s_i:t_i])_i$ be a point in $(\bfP^1)^g (k)$ and write $\ol
M_P$ for 
the $k$-subspace of $\ol M_0$ generated by $s_i  Y_i+t_i  X_i$ for
$i\in\Z/g\Z$. We have 
\[ F(s_{i-1}  Y_{i-1}+t_{i-1} X_{i-1})=
\begin{cases}
  t_{i-1}^pX_i & i\not\in \tau;\\
  t_{i-1}^p Y_i & i\in \tau. 
\end{cases} \]
From the closed condition $F\ol M_P\subset \ol M_P$ we get the
equations
\begin{equation}
  t_{i-1} s_i=0 \text{ for } i\not\in \tau, \quad\text{and }
  t_{i-1}t_i=0 \text{ for } i\in \tau.   
\end{equation}
From the closed condition $V\ol M_P\subset \ol M_P$ we get the
same equations as above. This finishes the computation.\qed
\end{proof}

\subsection{{\bf Examples.}}\label{34}
(1) If $\ul a=\ul 0$, then $\calX_\tau$ consists of two points:
    $([1:0],[1:0],\dots,[1:0])$ and $([0:1],[0:1],\dots,[0:1])$. 

(2) If $\ul a=(1,0,1,0,0)$, then $\calX_\tau$ is defined by the
    equations 
    $t_4t_0, t_0s_1, t_1t_2,t_2s_3,t_3s_4$. There are four irreducible
    components:
\begin{align*}
\bfP^1\times [0:1]\times [1:0]\times  [1:0]\times[1:0],\quad &
[1:0]\times 
[1:0]\times  \bfP^1\times [0:1]\times [0:1], \\     
[1:0]\times\bfP^1\times[1:0]\times\bfP^1\times[0:1],\quad
& [1:0] \times \bfP^1 \times [1:0] \times [1:0] \times \bfP^1.
\end{align*}
Notice that for maximally dimensional components, every $\bfP^1$ is
placed at a position $i$ where $a_i=0$.

(3) If $\ul a=(1,0,1,1,1,0)$, then $\calX_\tau$ is defined by the equations
    $t_5t_0,t_0s_1,t_1t_2, t_2 t_3,t_3 t_4,t_4s_5$. There are one
    3-dimensional component $[1:0]\times  \bfP^1\times [0:1]\times
    \bfP^1\times [0:1]\times  \bfP^1$, and four 2-dimensional
    components
\begin{align*}
\bfP^1\times [0:1]\times [1:0]\times  \bfP^1 \times [1:0]\times[1:0],
\quad & [1:0]\times  \bfP^1\times [1:0]\times [1:0]\times
\bfP^1\times [0:1], \\     
[1:0]\times[1:0]\times\bfP^1\times[1:0]\times\bfP^1\times[0:1],
\quad & 
[1:0]\times[1:0]\times\bfP^1\times[1:0]\times[1:0]\times\bfP^1.
\end{align*} 

\begin{prop}\label{35} \

{\rm (1)} We have $\dim \calX_\tau\le |\ul a|$. Furthermore, 
$\dim \calX_\tau=|\ul a|$ if and only if $\ul a\in \Delta^{\rm gen}$. 

{\rm (2)} For $\ul a\in \Delta^{\rm gen}$, the scheme 
$\calX_\tau$ has $w(\ul a)$ irreducible components of dimension 
$|\ul a|$.   
\end{prop}
\begin{proof}
  We may assume that $|\ul a|>0$, as the case $\ul a=\ul 0$ is treated
  in Example~\ref{34} (1). Since the defining equations are either
  $s_i=0$ or $t_i=0$, any irreducible component of $\calX_\tau$ is of the
  form $X=\prod_{i\in \Z/g\Z} X_i$, where 
\[ X_i=[1:0],\ [0:1],\ \text{ or } \bfP^1. \]
If $i\not\in \tau$, then we have $t_{i-1}s_i=0$. This tells us that there
are at least $g-|\ul a|$ zeros for $s_i$ or $t_i$ in the components
$[s_i:t_i]$ for $i\not\in \tau$ or $i-1\not\in \tau$. So $X_i=\bfP^1$
for at most $|\ul a|$ numbers of $i$. 
This shows that $\dim \calX_\tau\le |\ul a|$. 

If $\ul a\not\in \Delta^{\rm gen}$, then one can choose $i$ such that
$i-1\not\in \tau$, $i\in \tau$ and $i+1\in \tau$. It follows from the
equation $t_i t_{i+1}=0$ that there are at least $g-|\ul a|+1$ zeros
for $s_i$ or $t_i$ in
the components $[s_i:t_i]$ for $i\in \Z/g\Z$. Thus, $\dim
\calX_\tau<|\ul a|$. Suppose that  $\ul a\in \Delta^{\rm gen}$. Put $s_i=1$
for all $i\in \Z/g\Z$, then the defining equations become $t_{i-1}=0$
for $i\not\in \tau$. Thus, $\dim \calX_\tau=|\ul a|$. This proves the
statement (1). 

(2) Let $\ul a\in \Delta^{\rm gen}$ and 
    $X=\prod_{i\in \Z/g\Z} X_i$ be an 
    irreducible component of $\calX_\tau$. Write
    $\tau=\{n_1,\dots,n_a\}$. First 
    notice that 
    \begin{itemize}
    \item[(i)] If $X_{i_0}=\bfP^1$ for some $n_j\le i_0\le n_{j+1}$,
      then $X_i=[0:1]$ for $i_0<i<n_{j+1}$, and $X_i=[1:0]$ for $n_j\le
      i<i_0$ or $i_0<i=n_{j+1}$. 
    \end{itemize}
   It follows that
    \begin{itemize}
    \item[(ii)] There is at most one $i\in \Z$ in each interval
      $[n_j,n_{j+1}]$ such that $X_i=\bfP^1$. 
    \item [(iii)] If $X_i=\bfP^1$ for some $i\in \tau$, then $\dim
      X<|\ul a|$. 
    \end{itemize}
If $\dim X=|\ul a|$, then $X_{i_j}=\bfP^1$ for one $i_j$ in each
interval $n_j<i_j<n_{j+1}$. Conversely, choose $i_j$ in each interval
$n_j<i_j<n_{j+1}$. Then there is a unique
irreducible component $X$ such that $X_{i_j}=\bfP^1$ for each 
$j$; this follows from (i). There are $\prod_{j} (n_{j+1}-n_j-1)$ such
choices.  Thus, the 
scheme $\calX_\tau$ has $w(\ul a)$ irreducible components of
dimension $|\ul a|$.
\qed 
\end{proof}

Theorem~\ref{22} follows from Lemma~\ref{32} and
Proposition~\ref{35} (1).

\subsection{}{\bf Proof of Theorem~\ref{23}.}
\label{36}
Part (1) follows from Lemma~\ref{32} and Proposition~\ref{35} (2). We
prove the statement 
(2). We prove that irreducible components of $\calX_\tau$ give rise to
well-defined closed subvarieties in $\calM_{\Gamma_0(p),\ul a}$. 
Notice two
isomorphisms between $f^{-1}(x)_{\rm red}$ and $\calX_{\tau}$ (in
Lemma~\ref{32}) 
differ by an automorphism $\beta$ of $\ol M_0$, which sends each
factor of $(\bfP^1)^g$ to itself. If $X=\prod_i X_i$ is an irreducible
component of $\calX_\tau$, then the $i$-th component $\beta(X)_i$ of
$\beta(X)$ is equal to $\bfP^1$
whenever $X_i=\bfP^1$. By the property (i) in the proof of
Proposition~\ref{35}, we 
have shown that $\beta(X)=X$. Therefore, 
\[ \calM_X:=\{y\in \calM_{\Gamma_0(p),\ul a}\, |\, \xi_{f(y)}(y)\in
X\,\} \]
is a well-defined closed subvariety of $\calM_{\Gamma_0(p),\ul
  a}$. One has  $\calM_{\Gamma_0(p),\ul a}=\cup _X\calM_X$ as a union
of components; any irreducible component of $\calM_{\Gamma_0(p),\ul
  a}$ is contained in $\calM_X$ for a unique $X$. 
The morphism $f_{\ul a}:\calM_X\to \calM_{\ul a}$ is proper and
surjective with fibers isomorphic to $X$. Thus, $\Pi_0(
\calM_X)\simeq\Pi_0(\calM_{\ul a})$ and $\dim \calM_X=\dim \calM_{\ul
  a}+\dim X$. From this and Proposition~\ref{35}
(2) the statement (2) then follows. \qed

\subsection{} \label{37} 
{\bf Proof of Lemma~\ref{28}.} If $|\ul a|=j>0$, then
$w(\ul a)$ is the number of ways replacing a zero by 2 in $\ul a$ on
each interval $[n_j,n_{j+1}]$. In other words, $\sum_{|\ul a|=j}
w(\ul a)$ is
the number of ways of choosing $2j$ positions from $\Z/g\Z$ and 
filling them with $1$ and $2$ alternatively. This gives
$\sum_{|\ul a|=j} w(\ul a)=2
  \begin{pmatrix}
    g\\2j
  \end{pmatrix}$. Thus 
\[ \sum_{\ul a\in \Delta} w(\ul a)=2+\sum_{j>0}2
  \begin{pmatrix}
    g\\2j
  \end{pmatrix}=2^g. \]
This completes the proof. \qed

\subsection{} \label{38}
{\bf Proof of Theorem~\ref{212}.}
(1) The point $\eta$ lies in $\calM_{\Gamma_0(p),\ul a}$ for $\ul
a=\ul a(f(\eta))$, which is generic. Since any maximal point
of the generic alpha stratum $\calM_{\ul a}$ 
has slope sequence $s(|\ul a|,g)$ (Theorem~\ref{255}), the statement
follows. 

(2) Note that $s(j,g)$ is a non-supersingular slope sequence. 
  From (1) and Theorem~\ref{23} (2), 
  the number of maximal points of $\calM_{\Gamma_0(p)}$
  with slope sequence $s(j,g)$ is equal to $\sum_{|\ul a|=j} w(\ul
  a)$, which is $2
  \begin{pmatrix}
    g\\2j
  \end{pmatrix}$.

(3) This follows from Subsection~\ref{25}, Theorem~\ref{23} (2) and
    Theorem~\ref{26} (3).\qed
 
\section{Supersingular contribution}
\label{sec:04}

Keep the notation and the assumption of $p$ as before. 

\subsection{}\label{41} We recall the geometric mass formula for
superspecial abelian varieties of HB-type in \cite{yu:mass_hb}.

Let $x_0=\ul A_0=(A_0,\lambda_0,\iota_0,\eta_0)$ be a superspecial
(not necessarily separably) polarized abelian
$O_F$-variety over $k$ of dimension $g$ with symplectic level-$n$
structure with respect to $\zeta_n$. 
Let $\ul M_0=(M_0,\<\,,\>,\iota)$ be its covariant \dieu module with
additional structures. 
As $M_0$ is superspecial, the alpha type $\ul a$ of $\ul M_0$ has the
form 
\[ (e_1+e_2,\, 2-(e_1+e_2),\, e_1+e_2,\,\dots) \]
for some integers $e_1$, $e_2$ with $0\le e_1\le e_2\le 1$; see
\cite[Section 2]{yu:reduction}. When $g$ is odd, it satisfies the
additional condition $e_1+e_2=1$. We say that $\ul M_0$ is of {\it
  superspecial type} $(e_1,e_2)$ if its alpha type is as above. 

Let $G_{x_0}$ denote
the automorphism group scheme over $\Spec\, \Z$ associated to
$(A_0,\lambda_0,\iota_0)$ 
; for any commutative ring $R$, its group of
$R$-points is 
\[ G_{x_0}(R)=\{\phi\in (\End_{O_F}(A_0)\otimes R)^\times ; \phi'
\phi=1\}, \] 
where the map $\phi\mapsto \phi'$ is the Rosati involution induced by
$\lambda_0$.

Let $\Lambda_{x_0,n}$ denote the set of isomorphism classes of
polarized abelian $O_F$-varieties $\ul A=(A,\lambda,\iota,\eta)$ with
level-$n$ 
structure (w.r.t. $\zeta_n$) over $k$ such that (c.f. (2.4) of
\cite{yu:mass_hb}) 
\begin{itemize}
\item[(i)] the \dieu module $M(\ul A)$ is isomorphic to $M(\ul A_0)$,
  compatible with $O_F\otimes \Z_p$-actions and quasi-polarizations,
  and  
\item[(ii)] the Tate module $T_\ell(\ul A)$ is isomorphic to
  $T_\ell(\ul A_0)$, compatible with $O_F\otimes \Z_\ell$-actions and
  the Weil pairings, for all $\ell\not= p$.
\end{itemize}

The condition (i) implies that $A$ is superspecial and $\dim A=g$. 
Let $K_{n}$ be the kernel of the reduction map $G_{x_0}(\hat \Z)\to 
G_{x_0}(\hat \Z/n\hat \Z)$. There is a natural isomorphism 
\begin{equation}
  \label{eq:411}
  \Lambda_{x_0,n}\simeq G_{x_0}(\Q)\backslash G_{x_0}(\A_f)/K_{n};
\end{equation}
see \cite[Theorem 10.5]{yu:thesis} and \cite[Theorem 2.1 and 
Subsection 4.6]{yu:mass_hb}. 
It is proved in \cite[Theorem 3.7 and Subsection 4.6]{yu:mass_hb} that 
\begin{equation}
  \label{eq:412}
  |\Lambda_{x_0,n}|=[G_{x_0}(\hat \Z):K_{n}]\left [\frac{-1}{2}\right
   ]^g \zeta_F(-1) c_p, 
\end{equation}
where 
\begin{equation}
  \label{eq:413}
  c_p:=
  \begin{cases}
    1     & \text{$g$ is even and $e_1=e_2$,}\\
    p^g+1 &  \text{$g$ is even and $e_1<e_2$,}\\
    p^g-1 &  \text{$g$ is odd,}\\
  \end{cases}
\end{equation}
and $(e_1,e_2)$ is the superspecial type of $M_0$.

If $T_\ell(\ul A_0)\simeq (V\otimes \Z_\ell, \<\, ,\>,\iota)$ 
(Subsection~\ref{25}) for all
$\ell\neq p$, then it is easy to see that $[G_{x_0}(\hat
\Z):K_{n}]=[G(\Z):\Gamma(n)]$. In this case, the formula
(\ref{eq:412}) becomes 
\begin{equation}
  \label{eq:414}
    |\Lambda_{x_0,n}|=[G(\Z):\Gamma(n)]\left [\frac{-1}{2}\right
   ]^g \zeta_F(-1) c_p, 
\end{equation}
where $c_p$ is as above.

\subsection{}\label{42}
If $g$ is odd, then it follows from Theorem~\ref{255}
that $\calM_{\ul a}$ is
supersingular if 
and only if $|\ul a|=g$, that is, $\calM_{\ul a}$ consists of all
superspecial points in $\calM$. By the formula (\ref{eq:414}), 
we get the equation (\ref{eq:251}). 

If $g$ is even, then it follows from Theorem~\ref{255}
that 
$\calM_{\ul a}$ is supersingular if and only if $\ul a\preceq 
(1,0,\dots,1,0)$ or $\ul a\preceq  (0,1,\dots,0,1)$. If $|\ul a|=g$, 
then $\calM_{\ul a}$ consists of all superspecial points in
$\calM$. By the formula (\ref{eq:414}), we get the equation
(\ref{eq:252}). This proves the statements (1) and (2) of
Theorem~\ref{26}.  

\subsection{} \label{43} 
We prove Theorem~\ref{26} (3). 
Suppose $g=2d$ is even and $|\ul a|\not=g$. 
Put $\ul a_0:=(1,0,\dots,1,0)$. We may 
assume that $\ul a \preceq \ul a_0$ due to symmetry. 
Let $\calM^{(p)}$ be the moduli space over $\Fpbar$ of $g$-dimensional
separably polarized abelian $O_F$-varieties with a 
symplectic level-$n$ structure with respect to $\zeta_n$. 
We may identify the moduli space $\calM$ with an irreducible component
of  
$\calM^{(p)}$ by choosing an suitable element $\lambda_0\in L^+$; see
\cite[Proposition 4.1]{yu:mass_hb}.  

Choose any point $\ul A_0$ in $\calM_{\ul a_0}(k)$. Let $\ul M_0$ be the
covariant \dieu module of $\ul A_0$. Let $\ul N:=(F,V)M_0$, a \dieu
$\O$-submodule with the induced quasi-polarization.  Then there is a
tuple $\ul B=(B,\lambda_B,\iota_B,\eta_B)$ and an  $O_F$-linear
isogeny 
$\varphi:\ul B\to \ul A_0$ of a $p$-power degree, compatible with
additional structures, such that $M(B)= N\subset M_0$. 

One easily computes that $\ul N$ has alpha type $(0,2,\dots,0,2)$.
Then one can find a basis $\{X_i,Y_i\}$ for $N^i$ \cite[Lemma
4.4]{yu:thesis} such that
\begin{equation}
\begin{array}{lll}
FX_i=-pY_{i+1}, & FY_i=pX_{i+1}, & \text{if $i$ is even,} \\
FX_i=-Y_{i+1}, &  FY_i=X_{i+1}, & \text{if $i$ is odd}.
\end{array}
\end{equation}
Let $N_{-1}:=(F,V)^{-1}N$; it is spanned by elements
\[ \frac{1}{p} X_{2i}, \ \frac{1}{p} Y_{2i},  \ X_{2i+1},\ Y_{2i+1},
\quad i=0,\cdots, d-1. \]
We have $N_{-1}/N\cong k^2\oplus 0\oplus k^2 \oplus\cdots k^2 \oplus
0$ as $\O\otimes_{\Z_p} k$-modules. 
Let $\calX$ be the space of \dieu $\O$-modules $M$ such that 
\[ N\subset M \subset N_{-1},\  M/N\cong k\oplus 0\oplus k \oplus \cdots
k \oplus 0. \]
It is clear that $\calX\cong (\bfP^1)^d$.

Let $\Lambda$ denote the set of isomorphism classes of
objects $\ul B'=(B',\lambda',\iota',\eta')$ (with respect to $\zeta_n$)
such that (cf. Subsection~\ref{41}) 
\begin{itemize}
\item the \dieu module $M(\ul B')$ is isomorphic to  $M(\ul B)$,
  compatible with additional structures, and 
\item the Tate module $T_\ell(\ul B')$ is isomorphic to
  $T_\ell(\ul B)$, compatible with additional structures, for all
  $\ell\not= p$. 
\end{itemize}
\def\pr{\mathrm {pr}}

\begin{prop}\label{44}
  There is an isomorphism $\pr:\coprod_{\xi\in \Lambda} \calX\to \ol
  \calM_{\ul a_0}$. 
\end{prop}

\begin{proof}
  We write the map {\it set-theoretically} first. For any member
  $\xi\in 
\Lambda$ and any point $x\in \calX_\xi(k):=\calX(k)$, we have $M(\ul
B_{\xi})=\ul N\subset \ul M_x$. Then one gets a point $\ul A_x$
together with a polarized $O_F$-linear isogeny $\varphi:\ul B_{\xi}\to
\ul A_x$ of $p$-power degree such that $M(\ul A_x)=\ul M_x$. 
Define $\pr(\ul M_x):=\ul A_x$. Then one can show that it gives a
bijective map from 
$\coprod_{\xi\in \Lambda} \calX_\xi(k)$ onto $\ol \calM_{\ul
  a_0}(k)$. 
To see this map comes 
from a morphism of schemes, we need to construct a moduli space with
a prescribed isogeny type a priori, and show that this map agrees with
  the natural projection. 
Since the construction is lengthy and is the same
  as \cite[Lemma 9.1]{yu:thesis}, 
we refer the reader to that and omit the details here. 
Finally using the
  tangent space calculation, we prove that the morphism $\pr$ is
  {\'e}tale and particularly separable; 
  see the computation in Lemma~9.2 of \cite{yu:thesis}. Thus the
  morphism $\pr$ is isomorphism and the proof is complete.\qed  
\end{proof}

By definition $\Lambda$ is nothing but the set $\Lambda_{\ul B,n}$
defined in Subsection~\ref{41}. Note that the alpha type of $\ul B$ is
$(0,2,\dots,0,2)$. By the formula (\ref{eq:414}), we get  
\begin{lemma}\label{45}
  $|\Lambda|=[G(\Z):\Gamma(n)]\left [\frac{-1}{2}\right ]^g
  \zeta_F(-1)$. 
\end{lemma}

Denote by $\ol \calM_{\ul a_0,\xi}$ the irreducible component
 corresponding to $\xi$ and write  $\pr:\calX \to \ol \calM_{\ul
 a_0,\xi}$. Let ${\calM}_{ \preceq \ul a,\xi}\subset \ol {\calM}_{\ul
 a_0,\xi}$ be the closed subscheme consisting of points with alpha
 type $\preceq \ul a$.


\begin{lemma}\label{46}
The scheme ${\calM}_{\preceq \ul a,\xi}$ 
is isomorphic to $(\bfP^1)^{g-|\ul a|}$  
\end{lemma}
\begin{proof}
For a point 
$P=([x_0:y_0], [x_2:y_2], \cdots, [x_{2d-2}:y_{2d-2}])\in
(\bfP^1(k))^d$, the representing \dieu module is given by
\[ M_P=N+<\tilde x_{2i}\frac{1}{p} X_{2i}+ \tilde y_{2i}\frac{1}{p}
Y_{2i}>_{i=0,\cdots, d-1},\]
where $\tilde x_{2i}, \tilde y_{2i}$ are any liftings of $x_{2i},
y_{2i}$ in $W$, respectively. 

We compute the defining equations for 
$\calM_{\preceq  \ul a,\xi}$ on an affine open subset. Let $V_{2i}:=\tilde
x_{2i}\frac{1}{p} X_{2i}+ \frac{1}{p} Y_{2i}$, then
\[ M_P=<X_{2i},V_{2i},X_{2i+1},Y_{2i+1}>_{i=0,\cdots,d-1}, \text{ and }
M_P^{2i+1}=<X_{2i+1},Y_{2i+1}>\] 
One computes that
\[ ((F,V)M_P)^{2i+1}\  \left(\!\!\mod p M_P^{2i+1}\right)=\,<\ol
X_{2i+1}-x_{2i}^p \ol Y_{2i+1},\ -\ol X_{2i+1}+x_{2i+2}^{p^{-1}} 
\ol Y_{2i+1}>. \]
Therefore, $a_{2i+1}(M_P)=1$ if and only if $x_{2i}^{p^2}=x_{2i+2}$.

Let  $\tau=\tau(\ul a)\subset \Z/g\Z$.
We have shown that the subscheme ${\calM}_{\preceq  \ul a,\xi}$ of $\ol
\calM_{\ul a_0,\xi}=(\bfP^1)^{d}=\{(x_2,\cdots,x_{2d})\}$ 
defined by the equations $x^{p^2}_{j-1}=x_{j+1}$ 
for all odd $j\in \tau$, and thus it is
isomorphic to $(\bfP^1)^{g-|\ul a|}$. This completes the proof. \qed
\end{proof}

By Proposition~\ref{44} and Lemmas~\ref{45} and~\ref{46}, the
statement (3) of Theorem~\ref{26} is proved. 

\section{Unramified setting}
\label{sec:05}

In this section we only assume that $p$ is unramified in $F$.

\subsection{}\label{51}
Let \[ \calO:=O_F\otimes \Z_p,\quad \scrI:=\Hom (\calO,W),\quad
\Delta:=\{0,1\}^\scrI\] 
be the same as in Section~\ref{sec:02}. 
Let $\bbP$ be the set of primes of $O_F$ lying over $p$. For
$v\in\bbP$, let 
$\O_v$ be the completion of $O_F$ at $v$, $f_v$ its residue degree,  
$\scrI_v:=\Hom(\O_v,W)$ and $\Delta_v:=\{0,1\}^{\scrI_v}$. 
One has 
\[ \calO=\oplus_{v\in \bbP} \calO_v,\quad \sum_{v\in \bbP} f_v=g, \]
\[ \scrI=\coprod_{v\in \bbP} \scrI_v, \quad \scrI_v\simeq\Z/f_v\Z,
\quad \text{and $\Delta=\prod_{v\in \bbP} \Delta_v$}. \] 

An alpha type $\ul a=(\ul a_v)\in \Delta$ is called {\it generic} if 
every component $\ul a_v$ is generic; it is called {\it supersingular} 
if the associated alpha stratum $\calM_{\ul a}$ is supersingular.

Let $\Delta^{\rm gen}\subset \Delta$ be the subset of generic alpha
types, and $\Delta^{\rm gen}_{\rm ss}\subset \Delta^{\rm gen}$ the
subset of supersingular alpha types. The set $\Delta^{\rm gen}_{\rm
  ss}$ is empty if and only if $f_v$ is odd for some $v$. 

\begin{thm}\label{52} 
  For any alpha type $\ul a$, the set $\Pi_0(\calM_{\ul a})$ is
  $\ell$-adic Hecke transitive.   
\end{thm}
This is essentially due to Goren and Oort (cf.~\cite[Corollary
  4.2.4]{goren-oort}). We provide suitable details to fit the present
  situation: $p$ is unramified and the objects
  $(A,\lambda,\iota,\eta)$ that $\calM$ parametrizes may not be
  principally polarized abelian $O_F$-varieties.   

\begin{prop}\label{53} \ 

{\rm (1)} Every alpha stratum $\calM_{\ul a}$ is non-empty.

{\rm (2)} Every alpha stratum $\calM_{\ul a}$ is quasi-affine.   

{\rm (3)} The non-ordinary locus of $\calM$ is proper.

{\rm (4)} The Zariski closure $\ol \calM_{\ul a}$ of each stratum
  $\calM_{\ul a}$ in $\calM$ is smooth.

{\rm (5)} The set $\calM_{\ul 0}$ of superspecial points
  is $\ell$-adic Hecke transitive.  
\end{prop}
\begin{proof}
(1) It is easy to construct a superspecial point in
    $\calM$. Indeed, one constructs a separably polarized superspecial 
    abelian $O_F$-variety, then one chooses a point within its
    prime-to-$p$ isogeny class so that it lies in $\calM$. 
    Then one constructs a deformation of this point so that the
    generic point has the given alpha type $\ul a$. Such a construction
    is a local problem, and it reduces to inert cases. This proves the
    statement (1) 


(2) This is a global property; it does not follow directly from
    the result of inert cases. One can slightly modify the proof in
    \cite{goren-oort} to make it work. 
    Alternatively, consider the forgetful morphism
    $b:\calM\to \calA_{g,d,n}\otimes \Fpbar$, for some positive
    integer $d$ with $(d,p)=1$. Then the image $b(\calM_{\ul a})$ is
    contained in an Ekedahl-Oort stratum $S_\varphi$ of
    $\calA_{g,d,n}\otimes \Fpbar$. Since $S_\varphi$ is quasi-affine
    \cite{oort:eo}, the image $b(\calM_{\ul a})$ is also
    quasi-affine. Since the morphism $b$ is finite, the stratum
    $\calM_{\ul a}$ is quasi-affine.  

(3) This follows from the semi-stable reduction theorem for abelian
    varieties due to Grothendieck \cite{sga7_1}. 

(4) This is a local property, and hence follows directly from the
    results of inert cases.

(5) Let $x_0=(A_0,\lambda_0,\iota_0,\eta_0)$ be a superspecial point
    in $\calM$. Define $\Lambda_{x_0,n}$, $G_{x_0}$, $K_n$ as in
    Subsection ~\ref{41}. Note that any point in $\calM_{\ul 0}$
    satisfies the conditions (i) and (ii) in Subsection~\ref{41} (see
    \cite[Lemma~\ref{43}]{yu:thesis}). Therefore, we
    have $\Lambda_{x_0,n}=\calM_{\ul 0}$ and have the 
    double coset description
\begin{equation}
  \label{eq:521}
  \calM_{\ul 0} \simeq G_{x_0}(\Q)\backslash G_{x_0}(\A_f)/K_{n}
\end{equation}
as (\ref{eq:411}). By the strong approximation \cite[Theorem 7.12,
p.~427]{platonov-rapinchuk:agnt}, the natural map 
$G_{x_0}(\Q_\ell) \to G_{x_0}(\Q)\backslash G_{x_0}(\A_f)/K_{n}$ is
surjective. This shows that the $\ell$-adic Hecke orbit
$\calH_\ell(\ul A_0)$ is equal to $\calM_{\ul 0}$, and hence that 
the action of $\ell$-adic Hecke correspondences on the set 
$\calM_{\ul 0}$ is transitive. \qed
\end{proof}

\subsection{}{\bf Proof of Theorem~\ref{52}.}
Using (1), (2) and (3) of Proposition~\ref{53}, 
one shows that the closure of any irreducible
component $W$ of $\calM_{\ul a}$ contains a point in $\calM_{\ul
  0}$. By  Proposition~\ref{53} (4), 
any point in $\calM_{\ul 0}$ is contained in $\ol W$
for a unique irreducible component $W$ of $\calM_{\ul a}$. This shows
that there is a surjective $\ell$-adic Hecke equivariant map
\[ i: \calM_{\ul 0}\to \Pi_0(\ol \calM_{\ul a})=\Pi_0(\calM_{\ul
  a}). \]
By Proposition~\ref{53} (5), 
the set $\Pi_0(\calM_{\ul a})$ is $\ell$-adic Hecke
transitive. \qed   



An immediate consequence of Theorems~\ref{24} and~\ref{52}
is the following

\begin{cor} \label{55}
  Any non-supersingular stratum $\calM_{\ul a}$ is irreducible.
\end{cor}

\subsection{}\label{56}
Let $\ul a=(\ul a_v)_v\in \Delta$ be a supersingular alpha type. If
$f_v$ is odd, then $|\ul a_v|=f_v$. If $f_v$ is even, then either 
$\ul a_v \preceq  (1,0,\dots, 1,0)$ or $\ul a_v \preceq  (0,1,\dots, 0,1)$.
Define 
\[ \bbP_1:=\{v\in \bbP\, |\, \text{$f_v$ is odd}\, \} \]
\[ \bbP_2(\ul a):=\{v\in \bbP\, |\, \text{$f_v$ is even and $|\ul
  a_v|=f_v$}\, \} \] 
\[ \bbP_3(\ul a):=\{v\in \bbP\, |\, \text{$f_v$ is even and $|\ul
  a_v|<f_v$}\, \} \] 

\begin{thm}\label{57} Let $\ul a\in \Delta$ be a supersingular alpha
  type. Then 
  any irreducible component of $\ol \calM_{\ul a}$ is isomorphic to
  $(\bfP^1)^{g-|\ul a|}$ and  
  \begin{equation}
    \label{eq:571}
    |\Pi_0(\calM_{\ul a})|=[G(\Z):\Gamma(n)]\left 
    [\frac{-1}{2}\right ]^g
     \zeta_F(-1) \prod_{v\in \bbP} c_v, 
  \end{equation}
where 
\begin{equation}
  \label{eq:572}
  c_v:=
  \begin{cases}
    p^{f_v}-1   & \text{if $v\in \bbP_1$;}\\
    p^{f_v}+1   & \text{if $v\in \bbP_2(\ul a)$;} \\
    1 &  \text{if $v\in \bbP_3(\ul a)$.}\\
  \end{cases}
\end{equation}
\end{thm}

\begin{proof}
  Define $\ul a_0=(\ul a_{0,v})_v\in \Delta$ by 
\[ \ul a_{0,v}=
\begin{cases}
  (1,0,\dots,1,0) & \text{if $v\in \bbP_3(\ul a)$;}\\
  \ul a_v & \text{otherwise.}
\end{cases} \]
\end{proof}
We may assume that $\ul a \preceq  \ul a_0$ due to symmetry. Choose any
point $\ul A_0$ in $\calM_{\ul a_0}(k)$. Let $\ul M_0$ be the 
covariant \dieu module of $\ul A_0$. Define $N=\oplus N_v\subset
M_0=\oplus M_{0,v}$ the \dieu
$\O$-submodule with the induced quasi-polarization by
\[ N_v:=\begin{cases}
  (F,V)M_{0,v} & \text{if $v\in \bbP_3(\ul a)$;}\\
  M_{0,v} & \text{otherwise.}
\end{cases} \]
Then there is a tuple $\ul B=(B,\lambda_B,\iota_B,\eta_B)$ and an
$O_F$-linear isogeny $\varphi:\ul B\to \ul A_0$ of $p$-power degree,
compatible with additional structures, such that $M(B)= N\subset M_0$. 
Let $\calX=\prod_{v\in \bbP_3(\ul a)} \calX_v$, where $\calX_v$ is
  defined as $\calX$ in Subsection~\ref{43}. One has $\calX_v\simeq
  (\bfP^1)^{f_v/2}$.  
Define the set $\Lambda$ for $\ul B$ as in Subsection~\ref{43}. By
Proposition~\ref{44} we have an isomorphism 
$\coprod_{\xi\in \Lambda} \calX\simeq \ol \calM_{\ul a_0}$. Let
$\xi\in \Lambda$ and $\ol \calM_{\ul a_0,\xi}\simeq (\bfP^1)^{g-|\ul
  a_0|}$ be the corresponding component. By Lemma~\ref{46}, we show 
that $\ol \calM_{\ul a}\cap  \ol \calM_{\ul a_0,\xi}
\simeq (\bfP^1)^{g-|\ul a|}$. Therefore, we have
\begin{equation}
  \label{eq:573}
  \Pi_0(\ol \calM_{\ul a})\simeq \Pi_0(\ol \calM_{\ul
  a_0})\simeq \Lambda.   
\end{equation}

The alpha type of the factor $N_v$ is $(0,2,\dots, 0,2)$ if $v\in
\bbP_3(\ul a)$  and $(1,1,\dots)$ otherwise. Hence $N_v$ has
superspecial type
$(e_1,e_2)=(0,0)$ if $v\in \bbP_3(\ul a)$  and $(e_1,e_2)=(0,1)$
otherwise (Subsection~\ref{41}). By the mass formula 
\cite[Theorem 3.7 and
Subsection 4.6]{yu:mass_hb} (cf.~(\ref{eq:414})), we get
\[ |\Lambda|=[G(\Z):\Gamma(n)]\left [\frac{-1}{2}\right ]^g
     \zeta_F(-1) \prod_{v\in \bbP} c_v, \]
where $c_v$ is as above. This completes the proof.\qed

\subsection{}
\label{58}
Define the function $w':\Delta\to \R$ by
\begin{equation}
  \label{eq:581}
  w'(\ul a):=
  \begin{cases}
    [G(\Z):\Gamma(n)]\left [\frac{-1}{2}\right ]^g \zeta_F(-1) &
    \text{if $\ul a\in \Delta^{\rm gen}_{\rm ss}$;} \\
    \prod_{v\in \bbP} w(\ul a_v) & \text{otherwise,} \\
  \end{cases}
\end{equation}
where $w(\ul a_v)$ is the function as in (\ref{eq:221}). It is clear
that $w'(\ul a)\neq 0$ if and only if $\ul a\in \Delta^{\rm gen}$. It 
is rather unclear but indeed a fact that $w'(\ul a)\in \Z_{\ge
  0}$ (by (\ref{eq:592})).  

\begin{thm}\label{59}
Notation as above. We have 
\begin{equation}
  \label{eq:591}
  |\Pi_0(\calM_{\Gamma_0(p)})|=\sum_{\ul a\in \Delta^{\rm gen}} 
  w'(\ul a),  
\end{equation}
\end{thm}
\begin{proof}
  Suppose that $\ul a$ is a non-supersingular generic alpha type.
  It follows from the local computation in Section~\ref{sec:03} and 
  Corollary~\ref{55} that
  $\calM_{\Gamma_0(p), \ul a}$ has $w'(\ul a)$ irreducible components
  of dimension $g$.

  Suppose that $\ul a$ is a supersingular generic alpha type. Every
  fiber of the map $f_{\ul a}$ has one irreducible component of
  dimension $|\ul a|$ (Section~\ref{sec:03}). Thus,  
  $\calM_{\Gamma_0(p),\ul a}$ has $|\Pi_0(\calM_{\ul a})|$ irreducible
  components of dimension $g$. It follows from Theorem~\ref{57} that
  \begin{equation}\label{eq:592}
  |\Pi_0(\calM_{\Gamma_0(p),\ul a})|=
  |\Pi_0(\calM_{\ul a})|=w'(\ul a).
  \end{equation}
   This completes the proof. \qed
\end{proof}

We can rephrase Theorem~\ref{59} by an elementary combinatorial result 
(Lemma~\ref{28}) as follows 

\begin{thm}\label{510}
We have
\begin{equation}
  \label{eq:5101}
  |\Pi_0(\calM_{\Gamma_0(p)})|=2^g+\sum_{\ul a\in \Delta^{\rm
   gen}_{\rm ss}}  (w'(\ul a)-1), 
\end{equation}
\end{thm}

Similar to Theorem~\ref{212}, we have the following

\begin{thm}\label{511} \

{\rm (1)}  Let $\{s(j_v,f_v); v\in \bbP\}$ be a tuple of slope
sequences indexed by $\bbP$ and suppose that $0\le j_v< \lceil
\frac{f_v}{2}\rceil$ for some $v\in \bbP$. 
Then the moduli space $\calM_{\Gamma_0(p)}$ 
has exactly 
\[\prod_{v\in \bbP} 2 \begin{pmatrix} 
    f_v\\2j_v \end{pmatrix}\] 
irreducible components whose maximal point has
  slope sequence  $\{s(j_v,f_v); v\in \bbP\}$.

{\rm (2)} The moduli space $\calM_{\Gamma_0(p)}$ has  
\[ |\Delta^{\rm gen}_{\rm ss}|\cdot [G(\Z):\Gamma(n)]\cdot    
  \left [ \frac{-1}{2} \right ]^g \cdot \zeta_F(-1) \]
supersingular irreducible components.
\end{thm}

\begin{remark} \

(1) The connection of supersingular
strata with class numbers and special zeta values becomes a
standard fact now. If the moduli space $\calM_{\Gamma_0(p)}$ 
contains supersingular irreducible 
components, then it is expected that the special zeta value
$\zeta_F(-1)$ occurs in the formula for
$|\Pi_0(\calM_{\Gamma_0(p)})|$. 
However, the number of irreducible components of a
supersingular stratum is also related to $p$ in general.
It is indeed unexpected
that the number of supersingular  irreducible components of
$\calM_{\Gamma_0(p)}$ turns out  
to be independent of $p$. As a result, the number
$|\Pi_0(\calM_{\Gamma_0(p)})|$ of irreducible components is  
independent of $p$. 
We do not know any direct proof of this fact
without knowing the explicit formula (\ref{eq:5101}). \\


(2) The $p$-adic invariant stratification used in this paper
is nothing but the Ekedahl-Oort stratification 
(see \cite{oort:eo} and \cite{goren-oort}).  
Since a parahoric level structure on an abelian variety $A$ 
is a flag of finite flat subgroup schemes of the
$p$-torsion $A[p]$ that satisfy certain conditions, 
this structure only depends on the isomorphism class of $\ul A[p]$,
but not on $\ul A$. It would be interesting to know the relationship
between the group-theoretic description of Kottwitz-Rapoport strata
and that of Ekedahl-Oort strata by the works of Moonen
\cite{moonen:bt1,moonen:eo} and of Wedhorn \cite{wedhorn:eo}.\\

(3) The irreducibility problem for PEL-type moduli 
spaces $\calM_{K_p}$ with parahoric level structure at $p$ is, as
suggested by this work, related to the same problem for
Ekedahl-Oort strata in $\calM$ (without level at $p$), 
which is of interest in its own right. 
It seems plausible to expect that in any 
irreducible component of $\calM$,  
(i) any non-basic Ekedahl-Oort stratum is irreducible, and (ii) the
number of irreducible components of a basic Ekedahl-Oort stratum is a
{\it single} class number. 

For Siegel moduli spaces, the statement (i) is confirmed in Ekedahl and
van der Geer \cite{ekedahl-vdgeer:eo}, and the statement (ii) is
treated in Harashita \cite{harashita:sseo}. 

For Hilbert-Blumenthal moduli spaces, the statement (i) is essentially 
due to Goren and Oort \cite{goren-oort} and Chai \cite{chai:mono}
(Corollary~\ref{55}), 
and the statement (ii) is confirmed by Theorem~\ref{57} (see
(\ref{eq:573})).

\end{remark}

\section{$\ell$-adic monodromy of Hecke invariant subvarieties}
\label{sec:06}

The goal of this section is to provide a proof of a theorem of Chai on
Hecke 
invariant subvarieties for Hilbert-Blumenthal moduli spaces on which
Theorem~\ref{510} relies.
We follow the proof in Chai \cite{chai:mono} where the Siegel case is
proved. There is no novelty on the proof here and this is purely
expository; the author
is responsible for any inaccuracies and mistakes.
We write this as an independent section; some setup and notation may
be repeated and slightly modified.  
 
\subsection{}

Let $F$ be a totally real number field of degree $g$ and $O_F$ be the 
ring of integers in $F$. Let $V$ be a 2-dimensional vector space over
$F$ and $\psi:V\times V\to \Q$ be a $\Q$-bilinear non-degenerate
alternating form such that $\psi(ax,y)=\psi(x,ay)$ for all $x,y\in V$
and $a\in F$. Let $p$ be a fixed rational prime, not necessarily
unramified in $F$. We choose and fix an $O_F$-lattice
$V_{\Z}\subset V$ so that $V_{\Z}\otimes \Z_p$ is self-dual with
respect to $\psi$. We choose a projective system of
primitive prime-to-$p$-th roots of unity
$\zeta=(\zeta_m)_{(m,p)=1}\subset 
\Qbar\subset \C$. We also fix an embedding $\Qbar \hookrightarrow
\Qpbar$. 
For any prime-to-$p$ integer $m\ge 1$ and any connected
$\Z_{(p)}[\zeta_m]$-scheme $S$, the choice $\zeta$ determines an
isomorphism  
$\zeta_m:\Z/m\Z\isoto \mu_m(S)$, or equivalently, a $\pi_1(S,\bar
s)$-invariant $(1+m\hat \Z^{(p)})^\times$ orbit of isomorphisms $\bar
\zeta_m: \hat \Z^{(p)}\to \hat \Z^{(p)}(1)_{\bar s}$, where $\hat
\Z^{(p)}:=\prod_{\ell\neq p}\hat \Z_\ell$ and $\bar s$ is a geometric
point of $S$. 

Let $G$ be the automorphism group scheme over $\Z$ associated to the
pair $(V_\Z,\psi)$; for any commutative ring $R$, the group of
$R$-valued points is
\begin{equation}\label{eq:601}
  G(R):=\{g\in \GL_{O_F}(V_\Z \otimes_{\Z} R)\,;\,
\psi(g(x),g(y))=\psi(x,y), \ \forall\, x, y\in V_\Z \otimes_{\Z} R\,
\}.
\end{equation}

Let $n\ge 3$ be a prime-to-$p$ positive integer and $\ell$ be a prime
with $(\ell,pn)=1$ and $(\ell,{\rm disc}(\psi))=1$, where ${\rm
  disc}(\psi)$ is the discriminant of $\psi$ on $V_{\Z}$. 
Let $m\ge 0$ be a non-negative integer. Let $U_{n\ell^m}$ be the
kernel of the reduction map  $G(\hat \Z^{(p)})\to G(\hat
\Z^{(p)}/n\ell^m\hat \Z^{(p)})$; this is an open 
compact subgroup of $G(\hat \Z^{(p)})$.

Let $\calD=(F,V,\psi, V_\Z, \zeta)$ be a list of data as above.  
Denote by
${\calM}_{\calD,n\ell^m}$ the moduli space over
$\Z_{(p)}[\zeta_{n\ell^m}]$ 
that parametrizes equivalence classes of objects
$(A,\lambda,\iota,[\eta])_S$ over a connected locally Noetherian 
$\Z_{(p)}[\zeta_{n\ell^m}]$-scheme $S$, where
\begin{itemize}
\item $(A,\lambda)$ is a $p$-principally polarized abelian scheme over
  $S$ 
  of relative dimension $g$,
\item $\iota:O_F\to \End_S(A)$ is a ring monomorphism such that
  $\lambda\circ \iota(a)=\iota(a)^t \circ \lambda$ for all $a\in O_F$,
  and
\item $[\eta]$ is a $\pi_1(S,\bar s)$-invariant
  $U_{n\ell^m}$-orbit of $O_F$-linear isomorphisms 
  \begin{equation}
    \label{eq:611}
    \eta: V_\Z\otimes \hat\Z^{(p)}
  \isoto T^{(p)}(A_{\bar s}):=\prod_{p'\neq p} T_{p'}(A_{\bar s})
  \end{equation}
  such that 
  \begin{equation}
    \label{eq:612}
   e_\lambda(\eta(x),\eta(y))=\bar \zeta_{n\ell^m}(\psi(x,y))\
  ({\rm mod}\ (1+m\hat \Z^{(p)})^\times), \quad 
  \forall\, x, y\in  V_\Z\otimes \hat\Z^{(p)}, 
  \end{equation}
where $e_\lambda$ is the Weil pairing induced by the polarization
$\lambda$ and $\bar s$ is a geometric point of $S$. 
\end{itemize}

We write $[\eta]_{U_{n\ell^m}}$ for $[\eta]$ in order to specify the
level. 
Let $\calM_{n\ell^m}:=\calM_{\calD,n\ell^m}\otimes \Fpbar$ be the
reduction modulo $p$ of the moduli scheme $\calM_{\calD,n\ell^m}$. 
We have a natural morphism $\pi_{m,m'}:\calM_{n\ell^{m'}}\to
\calM_{n\ell^m}$, 
for $m<m'$, 
which sends $(A,\lambda,\iota,[\eta]_{U_{n\ell^{m'}}})$ to $
(A,\lambda,\iota,[ \eta]_{U_{n\ell^{m}}})$. 
Let $\wt \calM_n:=(\calM_{n\ell^m})_{m\ge 0}$ be the tower of this 
projective system.

Let $(\calX,\lambda,\iota, \bar \eta)\to \calM_n$ be the universal
family. The cover $\calM_{n\ell^m}$ over $\calM_{n}$ 
represents the \'etale sheaf 
\begin{equation}
  \label{eq:613}
   \calP_m:=\ul {Isom}_{ \calM_n}((V_\Z/\ell^m V_\Z,\psi),
(\calX[\ell^m],e_\lambda)\, ; \zeta_{\ell^m}) 
\end{equation}
of $O_F$-linear symplectic level-$\ell^m$ structures with respect to
$\zeta_{\ell^m}$. This is a $G(\Z/\ell^m\Z)$-torsor. 
Let $\bar x$ be a
geometric point in $\calM_n$. Choose an $O_F$-linear isomorphism
$y:V\otimes \Z_\ell\simeq T_{\ell}(\calX_{\bar x})$ that is
compatible with the polarizations with respect to $\zeta$. 
This amounts to choose a geometric point in 
$\wt \calM_n$ over the point $\bar x$. The action of the geometric
fundamental group 
$\pi_1(\calM_n,\bar x)$ on the system of fibers 
$(\calX_{\bar x}[\ell^m])_m$ gives rise to the monodromy
representation 
\begin{equation}
  \label{eq:614}
  \rho_{\calM_n,\ell}:\pi_1(\calM_n,\bar x)\to \Aut_{O_F}(
T_\ell(\calX_{\bar  x}),e_\lambda)
\end{equation}
and to the monodromy representation (using the same notation), 
through the choice of $y$,
\begin{equation}
  \label{eq:615}
  \rho_{\calM_n,\ell}:\pi_1(\calM_n,\bar x)\to G(\Z_\ell).
\end{equation}
\begin{lemma}\label{62}
  The map $\rho_{\calM_n,\ell}$ is surjective. 
\end{lemma}
\begin{proof}
It is well-known that 
$ \calM_{\calD,n\ell^m}(\C)\simeq  \Gamma(n\ell^m)\backslash
G(\R)/SO(2,\R)^g$, where  $\Gamma(n\ell^m):=\ker G(\Z)\to
G(\Z/n\ell^m\Z)$. It follows that the geometric generic fiber 
$\calM_{\calD,n\ell^m}\otimes \Qbar$ is connected. It follows from the 
arithmetic toroidal compactification constructed in Rapoport
\cite{rapoport:thesis}  
that the geometric special fiber $\calM_{n\ell^m}$ is
also connected. The connectedness of $\wt \calM_n$
confirms the surjectivity of $\rho_{\calM_n,\ell}$. \qed
\end{proof}

\subsection{}
The action of
$G(\Z_\ell)$ on $\wt \calM_n $ extends uniquely a continuous 
action of $G(\Q_\ell)$. Descending
from $\wt \calM_n$ to $\calM_n$, elements of  $G(\Q_\ell)$
induce algebraic correspondences on $\calM_n$, known as the
$\ell$-adic Hecke correspondences on $\calM_n$. More precisely, to each
$g\in G(\Q_\ell)$ we associate an $\ell$-adic Hecke correspondence
$(\calH_g,\pr_1,\pr_2)$ as follows. 
Extending isomorphisms $\eta$ to isomorphisms
\[ \eta': V\otimes \A^{(p)}_f\to V^{(p)}(A):=T^{(p)}(A)\otimes
\A^{(p)}_f,\]
we see the class $[\eta]_{U_n}$ gives rise to a class
$[\eta']_{U_n}$ in ${\rm Isom}(V\otimes \A^{(p)}_f, V^{(p)}(A))/U_n$
and $[\eta]_{U_n}$ is determined by $[\eta']_{U_n}$. 
The right translation $\rho_g: (A,\lambda,\iota,[\eta']_{U_n})\mapsto
(A,\lambda,\iota,[\eta'g]_{g^{-1}U_n g})$ gives rise an isomorphism
$\rho_g:\calM_n\simeq \calM_{g^{-1}U_n g}$. Let $U_{n,g}:=U_n\cap
g^{-1}U_n g$ and $\calH_{g}$ be the \'etale cover of $\calM_n$
corresponding to 
the subgroup $U_{n,g}\subset U_n$. Let $\pr_1$ be the natural
projection $ \calH_{g}\to\calM_n$ and $\pr_2:=\rho_g^{-1}\circ \pr: 
\calH_{g}\to\calM_n$ be the composition of the isomorphism 
$\rho_g^{-1}$ with the
natural projection $\pr:\calH_g\to \calM_{g^{-1}U_n g}$. This defines
an $\ell$-adic Hecke correspondence $(\calH_g,\pr_1,\pr_2)$. For two
$\ell$-adic Hecke correspondences 
$\calH_{g_1}=(\calH_{g_1},p_{11},p_{12})$ and 
$\calH_{g_2}=(\calH_{g_2},p_{21},p_{22})$, one defines the composition
$\calH_{g_2}\circ \calH_{g_1}$ by
\[ (\calH_{g_2}\circ \calH_{g_1}, p_1, p_2),
  \]
where $\calH_{g_2}\circ
\calH_{g_1}:=\calH_{g_1}\times_{p_{12},\calM_n, p_{21}} \calH_{g_2}$, 
$p_1$ is the composition $\calH_{g_2}\circ
\calH_{g_1}\to \calH_{g_1}\stackrel{p_{11}} {\to} \calM_n$ and $p_2$
is the composition 
$\calH_{g_2}\circ \calH_{g_1}\to \calH_{g_2}\stackrel{p_{22}} {\to}
\calM_n$. A correspondence
$(\calH,\pr_1,\pr_2)$ generated by correspondences of the form
$\calH_g$ is also called an $\ell$-adic Hecke correspondence.





A subset $Z$ of $\calM_n$ is
called {\it $\ell$-adic Hecke invariant} if $\rm{pr}_2
(\rm{pr}_1^{-1}(Z))\subset Z$ for any $\ell$-adic Hecke correspondence
$(\calH,\rm{pr}_1,\rm{pr}_2)$. If $Z$ is an $\ell$-adic Hecke invariant,
locally closed subvariety of $\calM_n$, then the $\ell$-adic 
Hecke correspondences induce 
correspondences on the set $\Pi_0(Z)$ of geometrically irreducible
components. We say that $\Pi_0(Z)$ is {\it $\ell$-adic Hecke
  transitive} if 
the $\ell$-adic Hecke correspondences operate transitively on
$\Pi_0(Z)$, that is, for any two maximal points $\eta_1, \eta_2$ of
$Z$ there is an $\ell$-Hecke Hecke correspondence
$(\calH,\rm{pr}_1,\rm{pr}_2)$ so that $\eta_2\in \rm{pr}_2
(\rm{pr}_1^{-1}(\eta_1))$. 
Let $k$ be an \ac field of \ch $p$.
For a geometric point $x\in \calM_n(k)$, denote by $\calH_\ell(x)$ the
$\ell$-adic Hecke orbit of $x$; this is the set of points generated by
$\ell$-adic correspondences starting from $x$.

\begin{lemma} \

{\rm (1)} For any point $x\in \calM_n(k)$, the corresponding abelian
variety 
$A_x$ is supersingular if and only if the $\ell$-adic Hecke orbit
$\calH_\ell(x)$ of $x$ is finite.

{\rm (2)} Any closed $\ell$-adic Hecke invariant subscheme $Z$ of
$\calM_n$ contains a supersingular point.
\end{lemma}
\begin{proof} (1) This is Lemma 7 in Chai \cite{chai:ho}. (2) This is
  Proposition 6 in Chai \cite{chai:ho}.

\end{proof}

\subsection{}\label{65}
Put $G_\ell:=G\otimes \Q_\ell$ (Subsection~\ref{eq:601}). One has
\[ G_\ell=\prod_{\lambda|\ell} G_\lambda, \quad
G_\lambda=\Res_{F_\lambda/\Q_\ell} \SL_{2,F_\lambda}. \]
Let $\pr_\lambda:G_\ell\to G_\lambda$ be the projection map.
Let $Z$ be a smooth locally closed subscheme of $\calM_n$ that is
$\ell$-adic Hecke invariant. Let $Z^0$ be a connected component of
$Z$, and $\eta$ be the generic point of $Z^0$. Let 
\[ \rho_{Z^0,\ell}:\pi_1(Z^0,\bar \eta)\to G(\Z_\ell) \]
be the associated $\ell$-adic monodromy representation, and
$\rho_{Z^0,\lambda}:=\pr_\lambda\circ \rho_{Z^0,\ell}$ be its
projection at $\lambda$.

\begin{lemma}\label{66} \

{\rm (1)} If the image ${\rm Im }\, \rho_{Z^0,\lambda}$ is finite
  for one $\lambda|\ell$, then the image ${\rm Im }\,
  \rho_{Z^0,\lambda}$ is 
  finite for all $\lambda|\ell$.

{\rm (2)} The abelian variety $A_\eta$ is not supersingular if and
only if  the image ${\rm Im }\, \rho_{Z^0,\lambda}$ is infinite for
all $\lambda|\ell$. 
\end{lemma}

\begin{proof}
  (1) Let $Z^0_0$ be a scheme over $\F_q$ such that $Z^0=Z^0_0\otimes
      \Fpbar$, and let $\eta_0$ be the generic point of
      $Z^0_0$. Replacing by a finite surjective cover of $Z^0_0$ (thus
      of $Z^0$), we may
      assume that $\End^0(A_{\bar
      \eta})=\End^0(A_{\eta_0}):=\End (A_{\eta_0})\otimes \Q$ and that
      ${\rm Im}\,\rho_{Z^0,\lambda}=1$ whenever it is finite. Write
      the Tate module 
      $V_\ell(A_{\bar \eta})=\prod_{\lambda|\ell} V_\lambda$ into
      the decomposition with respect to the action of $F$, and let
      $\rho_{\lambda}:\Gal(k(\bar \eta_0)/k(\eta_0))\to
      \Aut(V_\lambda)$ be 
      associated $\lambda$-adic Galois representation. Let $E_\lambda$
      be the $F_\lambda$-subalgebra of $\End_{F_\lambda}(V_\lambda)$
      generated 
      by the image $\rho_{\lambda}(\Gal(k(\bar \eta_0)/k(\eta_0))$. By
      a theorem of 
      Zarhin on endomorphisms of abelian varieties over function
      fields \cite{zarhin:end}, the subalgebra $E_\lambda$ is
      semi-simple and  
      the endomorphism algebra $\End^0_F(A)\otimes_F
      F_\lambda$ is isomorphic to the commutant of $E_\lambda$ in
      $\End_{F_\lambda}(V_\lambda)$. 
      If ${\rm Im}\,\rho_{Z^0,\lambda}=1$ for some $\lambda$, then
      $\rho_\lambda$ factors through the quotient $\Gal(\Fpbar/\F_q)$,
      and thus $E_\lambda$ is commutative. In this case,
      $\dim_{F_\lambda} \End^0_{F}(A)\otimes F_\lambda$ is $2$ or $4$,
      and the same that $\dim_F\End^0_F(A)$ is $2$ or $4$. This shows
      that the abelian variety $A_{\eta_0}$ is of CM-type. 
      By a theorem of Grothendieck on CM
      abelian varieties in \ch $p$ (\cite[p.~220]{mumford:av} and
      \cite[Theorem 1.1]{oort:cm}), $A_{\eta_0}$ is isogenous to, over
      a finite extension of $k(\eta_0)$, an abelian variety that is
      defined over a finite field. This shows the image ${\rm Im}\,
      \rho_{Z^0,\ell}$ is finite. Therefore, ${\rm Im }\,
      \rho_{Z^0,\lambda}$ is 
      finite for all $\lambda|\ell$.

   (2) It is proved in \cite[Corollary 3.5]{chai:mono} that $A_\eta$
       is not 
       supersingular if and only if  the image ${\rm Im}\,
       \rho_{Z^0,\ell}$ is infinite. The statement then follows from
       (1).\qed   
\end{proof}

\begin{lemma}\label{67}
  Let $H$ be a connected normal subgroup of an algebraic group
  $G_1\times \dots\times G_r$ over a field of \ch zero, where $G_i$
  is a connected simple algebraic group. Then $H$ is of the form
  $H_1\times \dots \times H_r$ with $H_i$ is $\{1\}$ or $G_i$. 
\end{lemma}
\begin{proof}
  See Section 9.4 in \cite{springer:lag}. \qed
\end{proof}

\begin{lemma}\label{68}
  Notation as in Subsection~\ref{65}, if the abelian variety $A_\eta$
  is not supersingular, then the image ${\rm Im}\, \rho_{Z^0,\ell}$ is
  an open subgroup of $G(\Z_\ell)$. 
\end{lemma}
\begin{proof}
  Replacing $Z$ by the orbit of the component $Z^0$ under all
  $\ell$-adic Hecke correspondences, we may assume that the set
  $\pi_0(Z)$ of connected components is $\ell$-adic Hecke
  transitive. Put 
  $M:={\rm Im}\, \rho_{Z^0,\ell}$ and let $H$ be the neutral component 
  of the algebraic envelope of $M$. It is proved in \cite[Proposition
  4.1]{chai:mono} that $M$ is open in $H(\Q_\ell)$ and 
  $H$ is a connected normal subgroup of
  $G_\ell$. By Lemma~\ref{67}, the group $H$ has the form
  $\prod_{\lambda|\ell}H_\lambda$ with $H_\lambda=\{1\}$ or
  $G_\lambda$. Since $A_\eta$ is not supersingular, it follows from
  Lemma~\ref{66} that $H=G$. This completes the proof. \qed
\end{proof}

\begin{lemma}\label{69} \
  Let $G$ be a connected simply-connected semi-simple
  algebraic group over a local field $K$ such that each simple factor
  of $G$ is $K$-isotropic. Then $G(K)$ has
  no proper subgroup of finite index.  
\end{lemma}

Lemma~\ref{69} follows immediately from the affirmative solution to the
Kneser-Tits problem for local fields (See Platonov \cite{platonov:kt} 
for more details). This is proved by 
Platonov \cite{platonov:kt} for the \ch zero case and by 
Prasad and Raghunathan \cite{prasad-raghunathan:kt} 
for the \ch $p$ case. Only the \ch zero case of Lemma~\ref{69} is
needed.






\begin{thm}[Chai]\label{611}
  Let $Z$ be an $\ell$-adic Hecke invariant, smooth locally closed
  subscheme of $M_n$. Let $\bar \eta$ be a geometric generic point of
  an irreducible component $Z^0$ of $Z$. Suppose that the abelian
  variety  
  $A_{\bar \eta}$ corresponding to the point $\bar \eta$ is not
  supersingular, and that the set $\pi_0(Z)$ of connected components 
  is $\ell$-adic Hecke transitive. Then the monodromy representation
  \begin{equation*}
    \rho_{Z^0,\ell}:\pi_1(Z^0, \bar \eta)\to G(\Z_\ell)
  \end{equation*}
is surjective and $Z$ is irreducible.       
\end{thm}
\begin{proof}
  Let $\wt Z^0$ and $\wt Z$ be the preimage in $\wt \calM_n$ of the
  subschemes $Z^0$ and $Z$, respectively, under the morphism $\pi:\wt
  \calM_n\to \calM_n$. Let $Y$ be a connected component of $\wt Z^0$
  and $M$ be the image ${\rm Im}\,\rho_{Z^0,\ell}$. 
The group $\Aut(Y/Z^0)$ of deck transformations is equal to $M$. Since
the group $G(\Z_\ell)$ acts transitively on the fiber $\pi^{-1}(x)$ for
any $x\in Z$ and $G(\Q_\ell)$ acts transitively on the set $\pi_0(Z)$,
the group $G(\Q_\ell)$ acts transitively on the set $\pi_0(\wt
Z)$. This gives a homeomorphism (see \cite[Lemma 2.8]{chai:mono})
\[ Q\backslash G(\Q_\ell)\isoto \pi_0(\wt Z), \quad g\mapsto g[Y], \]
where $Q$ is the stabilizer of the class $[Y]$ (in $\pi_0(\wt
Z)$). Clearly $Q\cap G(\Z_\ell)=M$ and we have $M\backslash
G(\Z_\ell)\simeq \pi_0(\wt Z^0)$. It follows from Lemma~\ref{68} 
that $\pi_0(\wt Z^0)= M\backslash G(\Z_\ell)$ is finite. 
Write $Z=\coprod_{i=0}^r Z_i$
as a disjoint union of connected components. Since $G(\Q_\ell)$ acts 
transitively on $\pi_0(Z)$ and $\pi_0(\wt Z^0)$ is finite, 
each $\pi_0(\wt Z_i)$ is finite. We have
\begin{equation*}
  \begin{split}
  |\pi_0(\wt Z)|<\infty \implies  | Q\backslash G(\Q_\ell)| <\infty
& \implies Q=G(\Q_\ell)\  (\text{by Lemma~\ref{69}})\\
& \implies M=G(\Z_\ell).  
  \end{split}
\end{equation*}
This shows the connectedness of $\wt Z$ and hence that of $Z$.
This completes the proof. \qed
\end{proof}



\

\end{document}